\documentclass[11pt]{amsart}
\usepackage{subfiles}
\usepackage{geometry}
\usepackage{amssymb,amsfonts,amsmath,longtable}
\usepackage[table]{xcolor}
\usepackage{fullpage}
\usepackage{enumitem}
\usepackage{tikz}
\usepackage{tikz-cd}
\usepackage{setspace}
\usepackage{bbm}
\usepackage{xcolor}
\usepackage{mathrsfs}

\usepackage[backend=biber,sorting=nyt]{biblatex}
\usetikzlibrary{arrows,positioning,decorations.pathmorphing,  decorations.markings} 
\setlist[enumerate]{topsep=1ex, itemsep=-0.5ex, parsep=1.5ex}
\setlist[itemize]{topsep=1ex, itemsep=-0.5ex, parsep=1.5ex}
\newtheorem{Thm}{Theorem}[section]
\newtheorem{Lem}[Thm]{Lemma}
\newtheorem{Prop}[Thm]{Proposition}
\newtheorem{Cor}[Thm]{Corollary}

\theoremstyle{definition}
\newtheorem{Def}[Thm]{Definition}
\newtheorem{Rem}[Thm]{Remark}

\theoremstyle{remark}

\newcommand\tab{\;\;\;\;\;\;}

\newcommand{\C}{\mathbb{C}}
\newcommand{\Z}{\mathbb{Z}}

\newcommand{\BO}{\mathbb{O}}
\newcommand{\BV}{\mathbb{V}}
\newcommand{\CM}{\mathcal{M}}
\newcommand{\CW}{\mathcal{W}}
\newcommand{\CA}{\mathcal{A}}
\newcommand{\CB}{\mathcal{B}}
\newcommand{\CH}{\mathcal{H}}
\newcommand{\CHC}{\mathcal{HC}}
\newcommand{\CZ}{\mathcal{Z}}
\newcommand{\CT}{\mathcal{T}}
\newcommand{\CN}{\mathcal{N}}
\newcommand{\CU}{\mathcal{U}}
\newcommand{\CP}{\mathcal{P}}

\newcommand{\SP}{\mathscr{P}}

\def \a{\alpha}
\def \b{\beta}

\newcommand{\w}{\omega}

\newcommand{\g}{\mathfrak{g}}
\newcommand{\fh}{\mathfrak{h}}

\newcommand{\n}{\mathfrak{n}}
\newcommand{\m}{\mathfrak{m}}

\newcommand{\upsi}{\underline{\psi}}
\newcommand{\uP}{\underline{P}}
\newcommand{\ulambda}{\underline{\lambda}}
\newcommand{\fb}{\mathfrak{b}}

\def \x{\times}
\def \<{\langle}
\def \>{\rangle}

\def\({\left(}
\def  \){\right)}

\newcommand{\HC}{\textnormal{HC}}
\newcommand{\Spec}{\textnormal{Spec}}
\newcommand{\Hom}{\textnormal{Hom}}
\newcommand{\End}{\textnormal{End}}
\newcommand{\pr}{\textnormal{pr}}
\newcommand{\bim}{\textnormal{-bim}}
\newcommand{\fin}{\textnormal{fin}}
\newcommand{\ad}{\textnormal{ad}}
\newcommand{\gr}{\textnormal{gr}}
\newcommand{\reg}{\textnormal{reg}}

\newcommand{\Eq}[1]{\begin{align}#1\end{align}} %Equations with numberings%
\newcommand{\Eqn}[1]{\begin{align*}#1\end{align*}}  % Equations without numberings%
\tikzset{->-/.style={decoration={
  markings,
  mark=at position .5 with {\arrow{latex}}},postaction={decorate}}}
\tikzset{
    >=latex
    }
\author{Trung Vu}
\title{On the functor relating  Harish-Chandra bimodules and  Soergel bimodules}
\address{Department
of Mathematics, Yale University, New Haven CT USA}
\email{trung.vu@yale.edu}

\begin{document}
\begin{abstract}In the 90's  Soergel constructed a functor that relates Harish-Chandra bimodules to Soergel bimodules. We revisit this functor and relate it to the restriction functor constructed by Losev between Harish-Chandra bimodules and bimodules over $W$-algebra associated to the regular nilpotent element. We compute the images of certain Harish-Chandra bimodules under the restriction functor and provide alternative proofs for many properties of the functor constructed by Soergel. Blocks of Harish-Chandra bimodules with integral central characters were studied in the works of Soergel and Stroppel. We will generalize results of Soergel and Stroppel in the case of integral blocks to  general blocks.
\end{abstract}

\maketitle
\section{Introduction}
In this paper the base field is $\C$. Let $\g$ be a semisimple Lie algebra, $G$ be its simply connected algebraic group, and $U(\g)$ be the universal enveloping algebra. For any $U(\g)$-bimodule $M$, we have the adjoint $\g$-action defined by $\ad_x m=xm-mx$ for any $x \in \g$ and $m \in M$.   A $U(\g)$-bimodule is called \emph{Harish-Chandra} if it is finitely generated as a bimodule and its adjoint $\g$-action is locally finite, i.e., every $m \in M$ is contained in a finite dimensional $\g$-stable subspace. Harish-Chandra bimodules  were first introduced by Harish-Chandra to  study the representations of complex algebraic groups viewed as  real reductive groups. The Harish-Chandra bimodules are also closely related to many other important objects in Representation theory such as category O \cite{BG80} or finite W-algebras \cite{IL11}. 

% see \cite{So92} for a brief discussion of this aspect and references therein

Let $Z$ be the center of $U(\g)$. In \cite{So92}, Soergel constructed a functor from the category of Harish-Chandra bimodules to the category of bimodules over $Z$. Statements about Harish-Chandra bimodules then can be translated into statements about $Z$-bimodules which are easier. In this paper, we will interpret Soergel's functor as the special case of the restriction functor in \cite{IL11} associated to the regular nilpotent element. We will give alternative ways to prove many properties of the functor in \cite{So92} by means of the restriction functor in \cite{IL11}. Furthermore, we will treat all blocks of the category of Harish-Chandra bimodules, including non-integral blocks which are not treated in \cite{So92} and \cite{STL}. Our main motivation behind this paper  is to use a similar approach  to study the Harish-Chandra bimodules for quantum groups at roots of unity. This should appear in the future work.

%Our main motivation is to use approaches in this paper to study the Harish-Chandra bimodules for quantum group at root of unity. Currently, there is no equivalence between the category of Harish-Chandra bimodules over quantum group at root of unity and the quantum category $O$. Moreover, one only expect the derived equivalence. We hope to construct such a derived equivalence by means of  a similar functor between the category Harish-Chandra bimodules over quantum group at root of unity  and the category of bimodules over the Harish-Chandra center of quantum group.  

To continue, we will need some notations. Let $\fh$ be a Cartan subalgebra of $\g$. Let $\Lambda \subset \fh^*$ be the weight lattice and $\Lambda_r \subset \Lambda$ be the root lattice of $\g$. Let $\Phi$ be the root system of $\g$ and $\rho$ be half the sum of positive roots. The dot action of the  Weyl group $W$ on $\fh^*$ is defined by  $w \cdot \lambda=w (\lambda+\rho)-\rho$ for any $\lambda \in \fh^*, w \in W$. Recall the Harish-Chandra isomorphism $Z \cong \C[\fh^*]^{(W, \cdot)}$. For any $\chi \in \Spec Z$, let $\m_\chi$ be the maximal ideal of $\chi$ in $Z$. Let $Z^{\wedge_\chi}$ be the completion of $Z$ at the maximal ideal $\m_\chi$. Let $U^\chi:=U(\g) \otimes_Z Z^{\wedge_\chi}$. For any $\lambda \in \fh^*$, let $\chi_\lambda$ denote the image of $\lambda$ under the natural map $\fh^*\rightarrow \fh^*/(W, \cdot)$. Whenever we write $\lambda \in \chi$, we view $\chi$ as a dot $W$-orbit in $\fh^*$. Let $\HC(U)$ denote the category of Harish-Chandra bimodules. Let $R:=\C[\fh^*]^{\wedge_0}$ be the completion of $\C[\fh^*]$ at the point $0 \in \fh^*$. 

%Let $\chi'$ denote the image of any weight $\lambda$  under the natural map $\fh^*\rightarrow \fh^*/(W, \cdot)$. Let  $\m_{\chi'}$ be the maximal ideal of $Z$ corresponding to $\chi'$. Let $Z^{\wedge_{\chi'}}$ be the completion of $Z$ at the maximal ideal $\m_{\chi'}$. In particular, we have $Z^{\wedge_0}\cong \C[\fh^*]^{\wedge_0}$, the completion of $\C[\fh^*]$ at $0$, denote this ring by $R$. The weight $\lambda$ in $\fh^*$ is \emph{dominant} if $\<\lambda+ \rho, \a^\vee\> \neq -1, -2, \dots $ for any roots $\a$. To any dominant weight $\lambda$, let $U^{\chi'}:=U(\g)\otimes_Z Z^{\wedge_{\chi'}}$. Let $HC(U)$ denote the category of Harish-Chandra bimodules.

\textbf{Completed Harish-Chandra bimodules:} Let $\chi, \chi' \in \Spec Z$. A  $(U^\chi, U^{\chi'})$-bimodule is called  \emph{Harish-Chandra} if it is finitely generated as a left module and its adjoint $\g$-action is locally finite. Note that for a $(U^\chi, U^{\chi'})$-bimodule with locally finite adjoint action of $\g$, being finitely generated as a left module is equivalent to being finitely generated as a right module, see Proposition $\ref{prop: equivalent axioms}$. Let $\HC(U^{\chi, \chi'})$ be the full subcategory of the category of $(U^\chi, U^{\chi'})$-bimodules consisting of all Harish-Chandra bimodules. 

Two points $\chi, \chi'$ in $\Spec Z$ are called \emph{compatible} if there are $\lambda, \lambda' \in \fh^*$ such that $\lambda-\lambda' \in \Lambda$ and $\chi=\chi_\lambda, \chi'=\chi_{\lambda'}$. Then the category $\HC(U^{\chi, \chi'})$ is zero unless $\chi$ and $\chi'$ are compatible, see Remark \ref{rem HCbim is zero}. A weight $\lambda \in \fh^*$ is \emph{dominant} if $\< \lambda+\rho, \a^\vee\> \neq -1, -2, \dots$ for each positive root $\a$. When $\chi$ and $\chi'$ are compatible, we can find a non-unique pair of dominant weights $(\mu, \lambda)$ such that $\mu \in \chi, \lambda \in \chi'$ and $\lambda -\mu \in \Lambda$. To such pair of dominant weights $(\mu, \lambda)$, we will introduce  the \emph{translation bimodule} $P^{\mu, \lambda}$ in $\HC(U^{\chi, \chi'})$, see Definition \ref{def translation bimod}.

% For any dominant weight $\lambda$, we have an algebra isomorphism $\epsilon_\lambda: Z^{\wedge_{\chi'}}\cong R^{W_\lambda}$ in $(\ref{identify complete Z})$, Section $\ref{ssec: Deformed Cat O}$.  Suppose $W_\lambda \subset W_\mu$ then we have an algebra embedding $Z^{\wedge_\chi}\xrightarrow[]{\epsilon_\mu}R^{W_\mu}\hookrightarrow R^{W_\lambda}\xrightarrow[]{\epsilon^{-1}_\lambda} Z^{\wedge_{\chi'}}$. This algebra map makes $Z^{\wedge_{\chi'}}$ into a $(Z^{\wedge_\chi}, Z^{\wedge_{\chi'}})$-bimodule in an obvious way. Similarly, if $W_\mu \subset W_\lambda$ then $Z^{\wedge_\chi}$ becomes a $(Z^{\wedge_\chi}, Z^{\wedge_{\chi'}})$-bimodule.

Let $C(G)$ be the center of $G$ and $\mathfrak{X}$ be the group of characters of $C(G)$. Let $Z\bim^{C(G)}$ be the category of $\mathfrak{X}$-graded $Z$-bimodules. Let $(Z^{\wedge_\chi}, Z^{\wedge_{\chi'}})\bim^{C(G)}$ be the category of $\mathfrak{X}$-graded $(Z^{\wedge_\chi}, Z^{\wedge_{\chi'}})$-bimodules. For a nilpotent orbit $O \subset \g$, Losev in \cite{IL11} constructed a functor $\bullet_\dag$ from $\HC(U)$ to a suitable category of equivariant bimodules over the $W$-algebra associated to $O$. In the case when $O$ is principal, the $W$-algebra is identified  with the center $Z$, and the functor becomes 
\[ \bullet_\dag: \HC(U)\rightarrow Z\bim^{C(G)}.\]
One can extend this functor to the functor
\begin{align}\label{restriction functor 1}\bullet_\dag: \HC(U^{\chi, \chi'})\rightarrow (Z^{\wedge_\chi}, Z^{\wedge_{\chi'}})\bim^{C(G)},
\end{align}
see Section \ref{ssec: complete-restriction-functor} for details. Let $(\mu, \lambda)$ be a pair of dominant weights such that $\mu \in \chi, \lambda \in \chi', \mu-\lambda \in \Lambda$.  Let $W_\mu, W_\lambda$ be the stabilizers of $\mu,\lambda $ under the dot $W$-action, respectively. Then there are algebra isomorphisms $\epsilon_\mu: Z^{\wedge_\chi}\cong R^{W_\mu}$ and $ \epsilon_\lambda: Z^{\wedge_{\chi'}}\cong R^{W_\lambda}$ as in \eqref{identify complete Z}. Suppose $W_\lambda \subset W_\mu$, then we will have an algebra embedding
\begin{equation}\label{eq Z inclusion} Z^{\wedge_\chi} \xrightarrow[]{\epsilon_\mu} R^{W_\mu} \hookrightarrow R^{W_\lambda} \xrightarrow[]{\epsilon^{-1}_\lambda} Z^{\wedge_{\chi'}}.
\end{equation}
This algebra morphism makes $Z^{\wedge_{\chi'}}$ into a $(Z^{\wedge_\chi}, Z^{\wedge_{\chi'}})$-bimodule in an obvious way. Similarly, if $W_\mu \subset W_\lambda$ then $Z^{\wedge_{\chi}}$ becomes a $(Z^{\wedge_\chi}, Z^{\wedge_{\chi'}})$-bimodule. Now we can state the first main result of this paper:
\begin{Thm}\label{Thm1}The functor $\bullet_\dag$ in $(\ref{restriction functor 1})$ has the following properties:
\begin{enumerate}[label=\alph*)]
\item  \label{Thm1a}It is exact and monoidal.
\item  \label{Thm1b} Let $P$ be any projective object in $\HC(U^{\chi, \chi'})$, and $M$ be any object in $\HC(U^{\chi, \chi'})$.  The following map is bijective:
\[ \Hom_{\HC(U^{\chi, \chi'})}(M,P)\rightarrow \Hom_{(Z^{\wedge_\chi}, Z^{\wedge_{\chi'}})\bim^{C(G)}}(M_\dag, P_\dag).\]
\item  \label{Thm1c} Suppose $W_\lambda \subset W_\mu$ then $P^{\mu, \lambda}_\dag\cong Z^{\wedge_{\chi'}}$ as $(Z^{\wedge_\chi}, Z^{\wedge_{\chi'}})$-bimodules, here $Z^{\wedge_{\chi'}}$ is equipped with the bimodule structure as above . Similarly, if $W_\mu \subset W_\lambda$ then $P^{\mu, \lambda}_\dag\cong Z^{\wedge_\chi}$.
\end{enumerate}
\end{Thm}

We will compare the construction in this theorem with the construction of Soergel in Section $1.3$. Let us make comments about the proof of Theorem $\ref{Thm1}$. Part $\ref{Thm1a}$ and $\ref{Thm1b}$ follow from the properties of the restriction functor $\bullet_\dag$ in \cite{IL11}. Let us sketch a proof of the first statement in part $\ref{Thm1c}$. The proof of the second statement is similar.  For any $M$ in $\HC(U^{\chi'', \chi})$ and any $N$ in $\HC(U^{\chi, \chi'})$, we have a functorial isomorphism $(M\otimes_{U^\chi} N)_\dag \cong M_\dag\otimes_{Z^{\wedge_\chi}} N_\dag$. Since $W_\lambda \subset W_\mu$,  we have the map $Z^{\wedge_\chi}\hookrightarrow Z^{\wedge_{\chi'}}$ as in \eqref{eq Z inclusion}. Therefore, the right $Z^{\wedge_{\chi'}}$-action on $P^{\mu, \lambda}$ gives us a right $Z^{\wedge_\chi}$-action on $P^{\mu, \lambda}$. The image $P^{\mu, \lambda}_\dag$ is computed by  the following two observations: The left $Z^{\wedge_\chi}$-action and the right $Z^{\wedge_\chi}$-action on $P^{\mu, \lambda}$ coincide, and $P^{\mu, \lambda}_\dag/ P^{\mu,\lambda}_\dag \m_{\chi'} \cong \C$. These observations will be proved by using the relations between the category $\HC(U^{\chi, \chi'})$ and the category $O$ as well as its deformation $\hat{O}$; see Proposition \ref{prop image of complete tranbimod} and Lemma \ref{lem image of complete tranbimod}. 

Note that by identifications $\epsilon_\mu: Z^{\wedge_\chi}\cong R^{W_\mu}$ and $ \epsilon_\lambda: Z^{\wedge_{\chi'}}\cong R^{W_\lambda}$, we can view the functor \eqref{restriction functor 1} as follows:
\[ \bullet_\dag: \HC(U^{\chi, \chi'}) \rightarrow (R^{W_\mu}, R^{W_\lambda})\bim^{C(G)}.\]

\subsection{Non-integral cases}

Let $\CP$ denote the subcategory of all projective objects in  $\HC(U^{\chi, \chi'})$. Let us fix a pair of dominant weights $(\mu, \lambda)$ such that $\chi=\chi_\mu, \chi'=\chi_\lambda,$ and $\mu-\lambda \in \Lambda$ as above. To such pair of dominant weights, we will introduce a full Karoubian additive subcategory $\CB_{\mu, \lambda}$ of the category $(R^{W_\mu}, R^{W_\lambda})\bim^{C(G)}$. The category $\CB_{\mu, \lambda}$ is just a variant of  the category of singular Soergel bimodules in \cite{GW} that we will formally define in Section $1.1.2$. Let $\widehat{\CH}:=K^b(\CB_{\mu, \lambda})$, the bounded homotopy category of the additive category $\CB_{\mu, \lambda}$. Using Theorem $\ref{Thm1}$, we can deduce the following theorem 
\begin{Thm}\label{Thm4}The functor $\bullet_\dag$ in $(\ref{restriction functor 1})$ induces an equivalence of additive categories:
\[ \CP\cong \CB_{\mu, \lambda},\]
 and it also induces an equivalence of triangulated categories:
\[ D^b(\HC(U^{\chi,\chi'}))\cong \widehat{\CH}.\]
\end{Thm}
In \cite{LY20}, one of the key results is to relate the category of monodromic sheaves on $U\backslash G\slash U$ to the suitable category of Soergel bimodules, see \cite[Proposition $9.5$]{LY20}. Theorem \ref{Thm4} can be thought as a $D$-module analog of this result over complex numbers.
 
 Let $W_{[\lambda]}=\{w\in W|w\cdot \lambda-\lambda \in \Lambda_r\}$ and $\widehat{W}_{[\lambda]}=\{w\in W|w\cdot \lambda-\lambda \in \Lambda\}$. The group $W_{[\lambda]}$ is called the \emph{integral Weyl group} of $\lambda$. In Section \ref{non integral weight},  a subgroup $C$ of $\widehat{W}_{[\lambda]}$ is constructed so that we have  a semi-direct decomposition $\widehat{W}_{[\lambda]}= C\ltimes W_{[\lambda]}$. Furthermore,  $w \cdot \lambda$ is dominant for any $w\in C$ and there is a natural inclusion of groups $C\hookrightarrow \Lambda/\Lambda_r$; particularly, $C$ is abelian.
\subsubsection{The category $\CB_{[\lambda]}$ associated to the group $\widehat{W}_{[\lambda]}$}  \label{ssec: Sbim at non integral }   We recall that $C(G)$ is the center of $G$ and  $\mathfrak{X}$ is  the group of its characters. Let $R\bim^{C(G)}$ be the category of $\mathfrak{X}$-graded $R$-bimodules. We now define two families of objects in $R\bim^{C(G)}$:
\begin{itemize}
    \item Let $\Pi_{[\lambda]}$ be the  collection of simple reflections in $W_{[\lambda]}$ as in the second bullet point of Section $\ref{non integral weight}$.  For any $s\in \Pi_{[\lambda]}$, we consider the $R$-bimodule $R\otimes_{R^s} R$ and equip it with the  trivial $C(G)$-action.
    \item To each $w\in C \subset \widehat{W}_{[\lambda]}$, we equip $R$ with an $R$-bimodule structure as follows: 
    \[f\cdot g=(wf)g,\tab g\cdot f=fg,\]
    for any $f,g \in R$. Let $R_w$ denote this $R$-bimodule and let $C(G)$ act on $R_w$ via the character  corresponding to the element $w \in C\subset \Lambda/\Lambda_r$.
\end{itemize}
\begin{Def}The category $\CB_{[\lambda]}$ is the smallest  additive subcategory of the category $R\bim^{C(G)}$ generated by  the following set:
\[\{R\otimes_{R^s}R, R_w\}_{s\in \Pi_{[\lambda]}, w\in C}\]
under taking tensor products, direct sums and direct summands.
\end{Def}
\subsubsection{The category $\CB_{\mu, \lambda}$ associated to the pair of dominant weights $(\mu, \lambda)$} Let us form the category $\CB_{[\lambda]}$   as above.  Let $M$ be any object of $(R^{W_\mu},R^{W_\lambda})\bim^{C(G)}$ then $M=\oplus_{\xi \in \mathfrak{X}} M_{\xi}$, here $M_{\xi}$ is the $\xi$-component of $M$. To any character $\xi'$ of $C(G)$, the new object $M[\xi']$ is defined by $M[\xi']_{\xi}=M_{\xi+\xi'}$ for any character $\xi$ of $C(G)$. 

 Let us  consider the restriction functor
\[
\mathcal{F}: R\bim^{C(G)}\rightarrow (R^{W_\mu},R^{W_\lambda})\bim^{C(G)}.
\]
 Let us consider the images of objects in $\CB_{[\lambda]}$ under this restriction functor. We then  shift the $\mathfrak{X}$-grading of all these images by the character corresponding to the coset $\mu-\lambda+\Lambda_r$. 
 \begin{Def}The  category $\CB_{\mu, \lambda}$ is the full Karoubian additive subcategory of the category $(R^{W_\mu},R^{W_\lambda})\bim^{C(G)}$ generated by these grading shifted images of objects in $\CB_{[\lambda]}$.
\end{Def}

%One thing we need to emphasize is that the role of the grading shift here is marginal since the grading shift induces an equivalence of categories. 

The grading shift in the definition of $\CB_{\mu, \lambda}$ is only needed to make sure that the $\mathfrak{X}$-grading on objects of $\CB_{\mu, \lambda}$ is compatible  with the functor $(\ref{restriction functor 1})$. Indecomposable objects in $\CB_{\mu, \lambda}$ are one-to-one correspondence with the double cosets $W_\mu\backslash \widehat{W}_{[\lambda]}/W_\lambda$, which is a corollary of Theorem $1$ in \cite{GW}. 

% Theorem $\ref{Thm4}$ says that  the category $D^b(\HC(U^{\chi, \chi'}))$ is governed by the group $\widehat{W}_{[\lambda]}$ and its double cosets $W_\mu\backslash \widehat{W}_{[\lambda]}/W_\lambda$.

\subsection{Comparison to Soergel's paper}
Let $\chi_0$ be the image of $0 \in \fh^*$ under the natural map $\fh^*\rightarrow \fh^*/(W, \cdot)$. Let $\CHC$ denote the full subcategory of the category $\HC(U)$ consisting of all objects $M$ such that $\m_{\chi_0}^n M=M\m_{\chi_0}^n =0$ for $n \gg 0$. Let $\CHC^n$ be the full subcategory of the category $\CHC$ consisting of all objects $M$ such that $M\m^n_{\chi_0}=0$. 

Let $S:=\C[\fh^*]$ which is naturally graded with $\fh^*$  in degree $1$. Let $\m$ be the maximal ideal at $0$ of $S$. Let $\Pi$ be the set of simple reflections in $W$. To each $s\in \Pi$, let $S^s$ be the $s$-invariant part of $S$. The category of Soergel bimodules is the full additive subcategory of the category of graded $S$-bimodules generated by $\{S\otimes_{S^s} S\}_{s\in \Pi}$ under taking tensor products, direct sums and direct summands. By \cite[Theorem 2]{So92}, the indecomposable Soergel bimodules are labeled by the elements of $W$. Let $B_x$ denote the indecomposable Soergel bimodule corresponding to the element $x\in W$, in particular $B_s=S\otimes_{S^s} S$. Under the natural inclusion $Z\hookrightarrow \C[\fh^*]=S$, any $S$-bimodule can be viewed as $Z$-bimodules.

 Soergel constructed the functor $\BV: \CHC \rightarrow Z\bim$  with the following properties, see Theorem $13$ and Proposition $11,12$ in \cite{So92}:
\begin{Thm} \label{Thm3}

a) $\BV$ is exact and monoidal.

b) For any $Q$ which is projective in $\CHC^n$  and any $M$ in $\CHC$, the following map is bijective:
\[ \Hom_{\CHC}(M,Q)\rightarrow \Hom_{Z\bim}(\BV M, \BV Q).\]

c) Let $P_x^n$ be the indecomposable projectives of $\CHC^n$, suitably parametrized by elements $x$ in the Weyl group $W$. Then $\BV P_x^n \cong B_x/B_x\m^n$.
\end{Thm}

In \cite[$\mathsection 3.3$]{So92}, Soergel introduced the category $O^{\infty}_0$ which can be viewed as a "thickened" version of the principal block $O_0$. He established a Beinstein-Gelfand type equivalence $BG: \CHC\xrightarrow[]{\cong} O^{\infty}_0$ and a functor $\BV: O^{\infty}_0\rightarrow Z\bim$. The functor $\BV: \CHC \rightarrow Z\bim$ can be viewed as the  composition $\CHC\xrightarrow[]{BG}O^{\infty}_0\xrightarrow[]{\BV} Z\bim$. This construction was used in \cite{So92} to give proofs for most of parts of Theorem \ref{Thm3} based on propeties of $\BV: O^\infty_0 \rightarrow Z\bim$. However,  the construction on its own cannot be used to show that $V$ is monoidal because $O_0^\infty$ does not carry an obvious monoidal structure.

 Soergel used another realization of the functor $\BV: \CHC\rightarrow Z\bim$ to establish its monoidal structure. There is an irreducible object in $\CHC$ which is  called the principal series $L$.  The subcategory $\CHC^n$ has enough projective objects  for any $n$. One can choose projective cover $P^n$ of $L$ in $\CHC^n$ for each $n$  such that we have a  system consisting of surjective morphisms $\dots P^{n+1}\twoheadrightarrow P^n \dots \twoheadrightarrow L$. Then $\BV(X)=\varprojlim\Hom_{\CHC}(P^n, X)$ for any $X$ in $\CHC$. Soergel then constructed a suitable system of morphisms $P^n\rightarrow P^n \otimes_{U(\g)} P^n$ which gives a monoidal structure on $\BV$.

 The category $\HC(U^{\chi_0,\chi_0})$ contains the category $\CHC$. Theorem $\ref{Thm3}$ can be deduced from Theorem $\ref{Thm1}$ and $\ref{Thm4}$ by considering the functor $\bullet_\dag: \HC(U^{\chi_0,\chi_0})\rightarrow Z^{\wedge_{\chi_0}}\bim^{C(G)}$. Note that $C(G)$ acts trivially on any objects of $\HC(U^{\chi_0,\chi_0})$, hence the functor $\bullet_\dag$ can be viewed as $\bullet_\dag: \HC(U^{0,0})\rightarrow Z^{\wedge_{\chi_0}}\bim$. The monoidal structure of the functor $\BV$ is then straightforward with the approach using the restriction functor as above.

% \subsection{Further remarks}\label{further remarks} In \cite{LY20}, one of the key results is to relate the category of monodromic sheaves on $U\backslash G /U$ to the suitable category of Soergel bimodules, see \cite[Proposition 9.5]{LY20}. Theorem \ref{Thm4} can be thought as a $D$-module analog of this result over complex number. However, the direct connection seems to be subtle. The reason is behind the complication of Riemann-Hilbert correspondence between monodromic sheaves and monodromic D-modules.
% \subsection{Connection to the other works} In \cite{LY20}, they mentioned that one can deduce the results for arbitrarily algebraically closed field from the results over the finite field. So there is a $\C$-version of theorem in \cite{LY20}. Our paper provides another proof for the theorem in the field $\C$. Need to talk about Groupoid...Localization...
\subsection{Organization of the paper}
In Section $\ref{sec: preliminaries}$ we collect some basic definitions, notations and study the group $\widehat{W}_{[\lambda]}$. Section $\ref{sec: Cat O and HC}$ recalls basic facts about the Harish-Chandra $U(\g)$-bimodules and the category O. Section $\ref{sec: Deform O and Complete HC}$ is about the deformed category $\hat{O}$ and the category of Harish-Chandra bimodules $\HC(U^{\chi, \chi'})$. Section $\ref{sec: restriction functor}$ is the main part of this paper where we study in detail the functor $\bullet_\dag$ in $(\ref{restriction functor 1})$ and prove  Theorem $\ref{Thm4}$.

\subsection{Acknowledgment} I am deeply grateful to Ivan Losev for suggesting this problem and for many fruitful discussions. I thank Pablo Boixeda Alvarez for useful conversations. I  thank Zachary Carlini for helpful comments which have greatly improved the exposition.  This work is partially supported by the NSF under grant DMS-2001139.

\section{Preliminaries} \label{sec: preliminaries}
In this section we establish notations and collect some facts that will be used later. Let $\g$ be a semisimple Lie algebra, $G$ be the  simply connected algebraic group and $U(\g)$ be the universal enveloping algebra. To simplify notations, we sometimes write $U$ instead of $U(\g)$. Let $\fh, \fb$ and $ \fb_{-}$ be a Cartan subalgebra, a Borel subalgebra and the opposite Borel subalgebra with respect to the pair $(\fh, \fb)$, respectively. Let $\n$ and $ \n_{-}$ be the nilradicals of $\fb$ and $ \fb_-$, respectively. We have the triangular decomposition $U(\g) \cong U(\n_-)\otimes U(\fh)\otimes U(\n)$.

Let $\Lambda$ and $\Lambda_r$ be the weight lattice and the root lattice in $\fh^*$, respectively. Let $\Phi$ be the set of roots in $\fh^*$. Let $\Phi_+$ and $\Phi_-$ be the set of positive roots and negative roots in $\fh^*$, respectively. Let $\a_1, \dots, \a_n$ be the set of simple positive roots in $\Phi_+$. 

Let $W$ be the corresponding Weyl group. Let $s_i$ be the simple reflection in $W$ corresonding to the simple root $\a_i$. The group $W$ acts naturally on $\fh$ and $\fh^*$. We also have the dot action of $W$ on $\fh^*$ defined by $w\cdot \lambda =w(\lambda+\rho)-\rho$, here $\rho$ is half the  sum of positive roots in $\fh^*$.

%On $\fh^*$, we have the partial order $>$, which is called \emph{dominant order}, defined as follows: $\lambda> \mu$ if $\lambda- \mu$ is an integral linear combination of positive roots with non-negative coefficients.

The weight $\lambda$ in $\fh^*$ is called \emph{regular} if $W \cdot \lambda$ is a free orbit. The weight $\lambda$ is \emph{dominant} if $\<\lambda+\rho, \a^\vee\> \neq -1, -2, \dots$ for each positive root $\a$. The weight $\lambda$ is \emph{anti-dominant} if $\<\lambda+\rho, \a^\vee\> \neq 1, 2, \dots$ for each positive root $\a$.

Let $Z$ denote the center of $U(\g)$. Recall the Harish-Chandra isomorphism 
\[ Z\cong \C[\fh^*]^{(W, \cdot)}.\]
For any $\lambda \in \fh^*$, let $\chi_\lambda$ be the image of $\lambda$ in $\fh^*/(W,\cdot)$ under the natural map $\fh^*\rightarrow \fh^*/(W,\cdot)$.

For any $\g$-module $V$, the vector space $\Hom_{\g}(V, U(\g))$ carries a natural $Z$-module structure as follows: $(z\cdot f)(v)=zf(v)$ for $v\in V, f\in \Hom_\g(V, U(\g))$ and $ z\in Z$. The $Z$-module $\Hom_\g(V, U(\g))$ is free by the  Kostant theorem \cite[Theorem 21]{K63}.

\begin{Thm}[Kostant theorem]
Let $V$ be an irreducible representation of $\g$. The $Z$-module $\Hom_{\g}(V, U(\g))$ is free and its rank equal to the multiplicity of the zero weight in $V$.
\end{Thm}

\subsection{Completions}
\

For any closed point $\chi$ in Spec $Z$, let $\m_\chi$ be its maximal ideal. Let $Z^{\wedge_\chi}$ be the completion of $Z$ at the ideal $\m_\chi$. Let  $U^{\wedge_\chi}$ be the completion of $U(\g)$ at the two-sided ideal $U(\g) \m_\chi$ and $U^\chi$ be the tensor product $U(\g) \otimes_Z Z^{\wedge_\chi}$. The adjoint $\g$-action on $U(\g)$ extends to $\g$-actions on $U^\chi$ and $U^{\wedge_\chi}$. There is a natural map:
\[\iota: U^\chi \rightarrow U^{\wedge_\chi}.\]
Since the $\g$-action on $U^\chi$ is locally finite, $\iota$ maps $U^\chi$ to $U^{\wedge_\chi}_{\fin}$, the $\g$-locally finite part of $U^{\wedge_\chi}$.

\begin{Lem}
Under the map $\iota$, the algebra $U^\chi$ maps isomorphically onto $U^{\wedge_\chi}_{\fin}$.
\end{Lem}

\begin{proof} It is enough to show that the natural map
\Eq{\label{Hom(V,.)}\Hom_\g(V, U^\chi)\rightarrow \Hom_\g(V, U^{\wedge_\chi})}
is bijective for any finite dimensional irreducible $\g$-representation $V$.

Let $\textnormal{Irr}(\g)$ be the set of irreducible representations of $\g$. The $\g$-action on $V$ and $Z$-module structure on $\Hom_\g(V, U(\g))$ give the $\g$-action and $Z$-module structure on the tensor product $V\otimes \Hom_{\g}(V, U(\g))$.  The natural isomorphism 
\[ \bigoplus_{V\in \textnormal{Irr}(\g)} V\otimes \Hom_\g(V, U(\g))\xrightarrow[]{\cong} U(\g)\]
is $\g$-linear and $Z$-linear. Therefore, we have
\Eqn{\Hom_\g(V, U^\chi)&\cong \Hom_\g(V, U(\g))\otimes_Z Z^{\wedge_\chi}\\
 \Hom_\g(V, U(\g)/\m_\chi^k)&\cong \Hom_\g(V, U(\g))/\m_\chi^k}

 Since $\Hom_\g(V, U(\g))$ is a free $Z$-module of finite rank by the Kostant theorem, we have
 \Eqn{\Hom_\g(V, U^\chi)&\cong  \Hom_\g(V, U(\g))\otimes_Z Z^{\wedge_\chi}\\
 &\cong \varprojlim \Hom_\g(V, U(\g))/\m_\chi^k\\
 &\cong \varprojlim \Hom_\g(V, U(\g)/\m_\chi^k)\\
 &\cong \Hom_\g(V, U^{\wedge_\chi}).}
This shows that \eqref{Hom(V,.)} is an isomorphism.

\end{proof}

\begin{Lem} $U^\chi$ is Noetherian.
\end{Lem}
\begin{proof}
    We have a surjective map of algebras: $Z^{\wedge_\chi}\otimes_\C U(\g)\twoheadrightarrow U^\chi$, therefore, it is enough to show that $Z^{\wedge_\chi}\otimes_\C U(\g)$ is Noetherian. By the PBW filtration , $\gr \; Z^{\wedge_\chi}\otimes_\C U(\g) \cong Z^{\wedge_\chi}\otimes_\C \C[\g^*]$. The latter algebra is Noetherian hence so is $Z^{\wedge_\chi}\otimes_\C U(\g)$.
\end{proof}

\subsection{Non-integral weights} \label{non integral weight}
\

When  we are dealing with the non-integral weights, it is necessary to consider the integral Weyl group instead of the Weyl group $W$. We now recall some facts about non-integral weights. For any $\lambda \in \fh^*$, let  $\Phi_{[\lambda]}$ denote $\{\a\in \Phi: \<\lambda, \a^\vee\> \in \Z\}$, the integral root system associated to $\lambda$. Let $E(\lambda) \subset \fh^*$ be the real subspace generated by $\Phi_{[\lambda]}$. The integral Weyl group $W_{[\lambda]}$ is the subgroup of $W$ generated by all reflections $s_\a$ for all $\a$ in $\Phi_{[\lambda]}$, equivalently, $W_{[\lambda]}$ is the set $\{w\in W|w\cdot \lambda-\lambda \in \Lambda_r\}$.  Let $W_\lambda$ be the stabilizer of $\lambda$ in $W$ under the dot action. The following results can be found in  \cite[$\mathsection 7.4$]{H08}.
\begin{itemize}
    \item If $\mu-\lambda \in \Lambda$ then $\Phi_{[\lambda]}=\Phi_{[\mu]}, W_{[\lambda]}=W_{[\mu]}$ . 
    \item The set $\Phi_{[\lambda]}$ forms a root system in $E(\lambda)$ and $W_{[\lambda]}$ is its Weyl group. There is a unique simple root system of $\Phi_{[\lambda]}$  in the set of positive roots $\Phi_+$. Let $\{\a^\lambda_1, \dots, \a^\lambda_k\}$ be this simple root system of $\Phi_{[\lambda]}$ and $\Pi_{[\lambda]}:=\{s_{\a^\lambda_1}, \dots, s_{\a^\lambda_k}\}$ be the set of  corresponding simple reflections in $W_{[\lambda]}$.
    \item For each $\lambda \in\fh^*$, there is a unique $\lambda^\sharp$ in $E(\lambda)$ such that $\<\lambda^\sharp, \a^\vee\>=\<\lambda, \a^\vee\>$ for all $\a \in \Phi_{[\lambda]}$.
    \item The weight $\lambda$ is dominant if and only if $\< \lambda+\rho, \a^\vee\> \geq 0$ for all $\a\in \Phi^+_{[\lambda]}:=\Phi_{[\lambda]}\cap \Phi^+$.
    \item The orbit $W_{[\lambda]}\cdot \lambda$ has a  unique dominant weight and a unique anti-dominant weight.
    \item For any $w\in W$, we have that $W_{[w\cdot \mu]}= W_{[w\mu]}=wW_{[\mu]}w^{-1}.$
\end{itemize}

% Let $\Phi^{\vee}_{[\lambda]}:=\{\a^\vee|\a\in \Phi_{[\lambda]}\}$ and $\Phi^{\vee, +}_{[\lambda]}:=\{\a^\vee|\a\in \Phi^+_{[\lambda]}\}$.
% \begin{Lem} 
% \begin{enumerate}[label=\alph*)]
% \item The subgroup $\widehat{W}_{[\lambda]}$ is contained in the normalizer of $W_{[\lambda]}$ in $W$. 
% \item The subgroup $\widehat{W}_{[\lambda]}$ acts on $\Phi^{\vee}_{[\lambda]}$. 
% \end{enumerate}
% \end{Lem}
% \begin{proof}
% \begin{enumerate}[label=\alph*)]
% \item Since $w\lambda-\lambda \in \Lambda$ then we have that  $W_{[w\lambda]}=W_{[\lambda]}$ but $W_{[w\lambda]}=wW_{[\lambda]}w^{-1}$. Therefore, $w$ is contained in the normalizer of $W_{[\lambda]}$ in $W$.
% \item Let $w\in \widehat{W}_{[\lambda]}$. We need to show that $w\a^{\vee}\in \Phi^\vee_{[\lambda]}$ for any $\a\in \Phi_{[\lambda]}$. Since $w^{-1}\lambda-\lambda \in \Lambda$ , it follows that $\< w^{-1}\lambda, \a^\vee\>\in Z$ for any $\a \in \Phi_{[\lambda]}$. So we have $\< \lambda, w\a^{\vee}\> \in Z$ which implies that $w\a^\vee \in \Phi^\vee_{[\lambda]}$. 
% \end{enumerate}
% \end{proof}

The next lemma concerns about the intersection of the coset $\lambda +\Lambda$ with the reflection hyperplanes in $\fh^*$
\begin{Lem}\label{subgeneric weights} There exists a dominant weight $\mu \in \lambda +\Lambda$ such that $\mu$ is regular. For any simple root $\a^\lambda_i$  in $\Phi_{[\lambda]}$, there exists a dominant weight $\mu_i\in \lambda +\Lambda$ such that $\textnormal{Stab}_W(\mu_i)=\{1, s_{\a^\lambda_i}\}$.
\end{Lem}
\begin{proof} Let $\delta=\sum_i m_i \w_i$ be a dominant weight in $\Lambda$. By choosing  $m_i \gg 0$ for all $i$, we can ensure that $\mu:=\lambda +\delta$ is a regular dominant weight in $\lambda +\Lambda$. The construction of $\mu_i$ is as follows:

\emph{Step $1$:} We will show that there is $\mu' \in \lambda+\Lambda$ such that $\<\mu', (\a^\lambda_i)^\vee\>=0$. Equivalently, there is $\mu''\in \Lambda$ such that $\< \mu'', (\a^\lambda_i)^\vee\>=\<\lambda, (\a^\lambda_i)^\vee\>$. Note that the right hand side is an integer . We can choose a set of simple roots  in $\Phi$ containing $\a^\lambda_i$, therefore, there is some $\w\in \Lambda$ such that $\< \w, (\a^\lambda_i)^\vee\>=1$. So we can choose $\mu''=\< \lambda, (\a^\lambda_i)^\vee\>\w$.

\emph{Step $2$:} We will show that there is $\delta \in \Lambda$ such that $\< \delta, (\a^\lambda_i)^\vee\> =0$ and $\< \delta, \b^\vee\> >0$ for any $\b\in \Pi_{[\lambda]}/\{\a^\lambda_i\}$. Let $\Lambda_{[\lambda]}$ be the weight lattice of $E(\lambda)$ corresponding the root system $\Phi_{[\lambda]}$.    Let choose $\delta'\in \Lambda_{[\lambda]}$ such that $\< \delta', (\a^\lambda_i)^\vee\>=0$ and $\< \delta', \b^\vee\>=1$ for all $\b\in \Pi_{[\lambda]}/\{\a^\lambda_i\}$. The restriction of the weight lattice $\Lambda$ on $E(\lambda)$ gives us a full rank sublattice of the weight lattice $\Lambda_{[\lambda]}$. Therefore, there is $m\in \Z_{\geq 0}$ such that $m \delta'$ belongs to the restriction of $\Lambda$ on $E(\lambda)$. Then we can choose $\delta$ to be  a lift of $m \delta'$ in $\Lambda$.

\emph{Step $3$:} We can choose $\mu_i=\mu'+l \delta$ for $l \gg 0$. 
\end{proof}
Let $\widehat{W}_{[\lambda]}:=\{w\in W|w\lambda-\lambda \in \Lambda\}$, equivalently, $\widehat{W}_{[\lambda]}=\{w\in W|w\cdot \lambda -\lambda \in \Lambda\}$. It is easy to see that $\widehat{W}_{[\lambda]}$ is a subgroup of $W$. Define the map $\tau: \widehat{W}_{[\lambda]} \rightarrow \Lambda/\Lambda_r$ as follows: $\tau(w)=w\lambda-\lambda +\Lambda_r$. Note that $w\cdot \lambda -\lambda +\Lambda_r=w\lambda -\lambda+\Lambda_r$.
\begin{Lem}
    The map $\tau$ is a group homomorphism. The kernel of $\tau$ is $W_{[\lambda]}$. As a consequence, we have an  inclusion $\widehat{W}_{[\lambda]}/W_{[\lambda]}\hookrightarrow \Lambda/\Lambda_r$ and $\widehat{W}_{[\lambda]}/W_{[\lambda]}$ is an abelian group.
\end{Lem}
\begin{proof}
    Let $w_1, w_2 \in \widehat{W}_{[\lambda]}$. We have 
    \begin{align*}
        \tau(w_1w_2)-(\tau(w_1)+\tau(w_2))&=(w_1w_2(\lambda)-\lambda+\Lambda_r)-(w_1\lambda-\lambda+w_2(\lambda)-\lambda+\Lambda_r)\\
        &=w_1(w_2\lambda-\lambda)-(w_2\lambda-\lambda)+\Lambda_r\\
        &=\Lambda_r
    \end{align*}
The last equality holds since $w_2\lambda-\lambda \in \Lambda$. So $\tau$ is a group homomorphism. 

The element $w \in \widehat{W}_{[\lambda]}$ is in the kernel of $\tau$ if and only if $w\lambda-\lambda \in \Lambda_r$. This means that $W_{[\lambda]}$ is the kernel of $\tau$.
\end{proof}

The dot action of $W$ on $E$ induces an action of $\widehat{W}_{[\lambda]}$ on $E(\lambda)$ which consists of Euclidean transformations. The group $W_{[\lambda]}$ is the Weyl group of the root system $\Phi_{[\lambda]}$ in $E(\lambda)$, so we can talk about the dominant Weyl chamber under the dot action of $W_{[\lambda]}$. Let $C$ be the subgroup of $\widehat{W}_{[\lambda]}$ consisting of all elements that map the dominant Weyl chamber of $E(\lambda)$ to itself. The following lemma follows by standard results about Euclidean transformations on Euclidean spaces:
\begin{Lem} We have a semi-direct decomposition $\widehat{W}_{[\lambda]} =C\ltimes W_{[\lambda]}$; hence, we have a natural embedding $C\hookrightarrow \Lambda/ \Lambda_r$. If $\lambda$ is a dominant weight,  then $w \cdot \lambda$ is a dominant weight for any $w\in C$. 
\end{Lem}
\section{Category O and Harish-Chandra bimodules}\label{sec: Cat O and HC}

\subsection{Category $O$} \label{ssec: Cat O}\

Category $O$ is the full subcategory of the category of finitely generated $U(\g)$-modules such that the action of $\fh$ on each object in $O$ is semisimple and the action of $\n$ on each object in $O$ is locally nilpotent. We will recall some facts about the category $O$ based on \cite{H08}. 

Category $O$ is Artinian, i.e., each object in $O$ has finite length. All irreducible modules in $O$ are of the form $L(\lambda)$, the unique irreducible quotient of the Verma module $\Delta(\lambda)$. If $V$ is a finite dimensional $\g$-representation and $M$ is an object in $O$ then $V\otimes M$ also belongs to $O$.

 %For each $\lambda$ let $O_{\lambda+\Lambda_r}$ be the full subcategory of $O$ consisting of all objects whose weights belong to $\lambda+\Lambda_r$. Under the action of $\fh$, we have the decomposition of categories $O=\bigoplus_\lambda O_{\lambda+\Lambda_r}$ where $\lambda$ runs over all representatives of the cosets in $\fh^*/\Lambda_r$. Note that the intersection of each $W$-orbit in $\fh^*$ under the dot action with each $\Lambda_r$-coset in $\fh^*$ is of the form $W_{[\lambda]}\cdot \lambda$ for a unique dominant weight $\lambda$.
 
 Let us recall the root lattice $\Lambda_r$ in $\fh^*$. Let $O_\lambda$ be the full subcategory of $O$ consisting of all objects whose weights lie in $\lambda+\Lambda_r$ and whose supports in Spec $Z$ are $\{\chi_\lambda\}$. The category $O$ decomposes into $\bigoplus_\lambda O_\lambda$ where $\lambda$ runs over all dominant weights in $\fh^*$. Let $\pr_\lambda: O\rightarrow O_\lambda$ be the natural projection. Both $O$ and $O_\lambda$ have enough projectives. If $\lambda$ is dominant then $\Delta(\lambda)$ is a projective object in both $O_\lambda$ and $O$. All simple objects of $O_\lambda$ are of the form $L(w\cdot \lambda)$ for some $ w\in W_{[\lambda]}$.

 Let $O_{\chi}=\bigoplus O_{\lambda'}$ where $\lambda'$ runs over all dominant weights in the $W$-orbit $\chi$. Then $O_{\chi}$ is the full subcategory of the category $O$ consisting of all objects whose supports in Spec $Z$ are $\{\chi\}$. We have a direct sum decomposition $O=\bigoplus O_{\chi}$. If $\lambda$ is integral and dominant then $O_{\chi_\lambda}=O_\lambda$. Let $\pr_{\chi}: O\rightarrow O_{\chi}$ be the projection.

Let $\lambda$ and $\mu$ be two dominant weights such that $\mu-\lambda\in \Lambda$. Let $L(\mu-\lambda)$ be the finite dimensional irreducible $\g$-representation whose the highest weight is contained in the orbit $W(\mu-\lambda)$.

\begin{Def} \label{def translation functor}

a) For any finite dimensional $\g$-representation $V$, the \emph{projective functor} $\CP_{\chi'\rightarrow \chi}^V:O_{\chi'} \rightarrow O_{\chi}$ is defined by $\pr_{\chi}(V\otimes M)$ for any $M$ in $O_{\chi'}$.

b) Suppose there is a pair of dominant weights $(\lambda, \mu)$ such that $\mu \in \chi, \lambda\in \chi'$ and $\mu-\lambda \in \Lambda$. The \emph{translation functor} $\CT_{\lambda \rightarrow \mu}: O_{\chi'}\rightarrow O_{\chi}$ is defined by $\CP_{\chi'\rightarrow \chi}^{L(\mu-\lambda)}$.

\end{Def}

The projective functor $\CP^V_{\chi'\rightarrow \chi}$ sends projective objects in $O_{\chi'}$ to projective objects in $O_{\chi}$. We will need the following lemma about the translation functor , see \cite[$\mathsection 7.6$]{H08}. 
\begin{Lem}\label{lem image of Verma} Let $(\mu, \lambda)$ be a pair of dominant weights such that $ \mu \in \chi, \lambda \in \chi'$ and $\mu-\lambda \in \Lambda$. Assume that $W_\lambda\subset W_\mu$. Then $\CT_{\lambda\rightarrow \mu}(\Delta(w\cdot \lambda))=\Delta(w\cdot \mu)$ for any $w\in W_{[\lambda]}$.
\end{Lem}

The module $L(\mu-\lambda)\otimes \Delta(w\cdot \lambda)$ has a filtration whose successive quotients are $\Delta(w\cdot \lambda+\nu)$ in which $\nu$ runs over all weights of $L(\mu-\lambda)$ counted with multiplicities. Lemma \ref{lem image of Verma} follows from two observations: $(1)$ $ w\cdot \lambda+\nu \subset W \cdot \mu \Leftrightarrow w\cdot \lambda+\nu \in W_{[\lambda]}\cdot \mu \Leftrightarrow \nu=w(\mu-\lambda)$; and $(2)$ $ \dim_\C L(\mu-\lambda)_{w(\mu-\lambda)}=1$. The first observation is a consequence of Lemma $7.5$ in \cite{H08} and the second observation is trivial.

%In general, the projective functors do not restrict to functors from $O_\lambda$ to $O_\mu$ but for the functor $\CT_{\lambda\rightarrow \mu}$, it does.
%\begin{Cor}The functor $\CT_{\lambda\rightarrow \mu}$ induces a fucntor $\CT_{\lambda\rightarrow \mu}: O_\lambda\rightarrow O_\mu.$
%\end{Cor}

\subsection{Harish-Chandra bimodules}\

On any $U(\g)$-bimodule $M$, we have the adjoint $\g$-action defined by $\ad_xm=xm-mx$ for any $x\in \g$ and $m \in M$. A bimodule $M$ is called Harish-Chandra if it satisfies the following axioms:
\begin{enumerate}[label={(\bfseries R\arabic*)}]
\itemsep-0.3em 
\item The adjoint $\g$-action on $M$ is locally finite.
\item  \label{R2} $M$ is finitely generated as a bimodule.
\end{enumerate}

Let $\HC(U)$ denote the full subcategory of the category of $U(\g)$-bimodules consisting of all Harish-Chandra bimodules. We now go over some facts about $\HC(U)$ based on \cite{BG80}.

For each finite dimensional $\g$-representation $V$, the tensor product $V\otimes U(\g)$ has a natural $U(\g)$-bimodule structure as follows
\[ x(v\otimes u)=xv\otimes u+v\otimes xu, \tab (v\otimes u)x=v\otimes ux,\]
for $x\in \g, v\in V$ and $u\in U(\g)$. Similarly, we can equip $U(\g)\otimes V$ with a $U(\g)$-bimodule structure by 
\[ x(u\otimes v)=xu \otimes v, \tab (u\otimes v)x=ux\otimes v-u\otimes xv.\]
With these bimodules structures, both $U(\g)\otimes V $ and $V\otimes U(\g)$ become Harish-Chandra bimodules. Moreover, they are isomorphic to each other via a natural isomorphism:
\[ V\otimes U(\g) \rightarrow U(\g) \otimes V: \tab v\otimes u\mapsto \sum u_{(1)}\otimes S(u_{(2)})v,\]
where $S$ is the antipode of the Hopf algebra $U(\g)$ and $\Delta(u)=\sum u_{(1)}\otimes u_{(2)}$ is the comultiplication of $u$; here we use the Sweedler notation for the coproduct. For finite dimensional $\g$-modules $V$ and $V'$, we have an isomorphism of Harish-Chandra bimodules
\[(V\otimes U(\g))\otimes_{U(\g)}(V'\otimes U(\g)) \cong (V\otimes V')\otimes U(\g).\]

For any $M$ in $\HC(U)$, the space $\Hom_\g(V, M)$ is a module over $\CZ:=Z\otimes Z$ the center of $U(\g) \otimes U(\g)^{\textnormal{op}}$. Let $M^{\ad}$ be the space $M$ viewed as a $\g$-module under the adjoint $\g$-action.
\begin{Prop}\label{prop property of HCbim}
a) $\Hom_{\HC(U)}(V\otimes U(\g), M)=\Hom_\g(V, M^{\ad})$. The bimodule $V\otimes U(\g)$ is projective in $\HC(U)$.

b) Any Harish-Chandra bimodule is the quotient of $V\otimes U(\g)$ for some finite dimensional $\g$-representation $V$.

c) The space $\Hom_\g(V, M^{\ad})$ is finitely generated as a left $Z$-module and as a right $Z$-module for any finite dimensional  $\g$-representation $V$.

\end{Prop}

\begin{Rem} \label{rem support of HCbim}The support of $V\otimes U(\g)$ over Spec $\CZ$ is the image of 
\[ \mathcal{S}=\cup_{\nu, V_{\nu}\neq 0}\{(\lambda, \mu)|\mu-\lambda =\nu\}\]
under the natural map $\fh^* \x \fh^* \rightarrow \fh^*/(W, \cdot) \x \fh^*/(W, \cdot)=\Spec \;\CZ$, see Theorem $2.5$ in \cite{BG80}. Therefore, the projections from the support of $V\otimes U(\g)$ to the two copies of $\fh^*/(W, \cdot)$ are finite morphisms.
\end{Rem}

% for any Harish-Chandra bimodule $M$ and for any $\lambda \in \fh^*$, there are finite  subsets $S_\lambda^l$ and $S_\lambda^r$ of  $\fh/(W, \cdot)$ such that the support of $\Hom_\g(V, M)$ in $\fh^*/(W, \cdot) \x \fh^*/(W, \cdot)$ equals to  $\bigcup_{\lambda\in \fh^*} \chi' \x S^l_\lambda$  and equals to $\bigcup_{\lambda\in \fh^*}S^r_\lambda \x \chi'$.

The following result is Corrollary $5.7$ in \cite{BG80}:
\begin{Prop} \label{prop char of HCbim}
a) Let $J\subset \CZ$ be an ideal of finite codimension. Then there are only finitely many simple Harish-Chandra bimodules $L$ such that $JL=0$.

b) Let $Y$ be a $U(\g)$-bimodule such that  the adjoint $\g$-action on $Y$ is locally finite. If $Y$ is annihilated by an ideal $J$ in $\CZ$ of finite codimension and $\Hom_\g(V,Y)< \infty$ for any finite dimensional $\g$-representation $V$, then $Y$ is a Harish-Chandra bimodule of finite length.

\end{Prop}

%To the end of this section, let $\mu, \lambda$ be  dominant weights in $\fh^*$ such that $\mu-\lambda \in \Lambda$. We recall the corresponding closed point $\chi'$ in Spec $Z$ and the maximal ideal $\m_{\chi'}$ in $Z$. We recall the Verma module $\Delta(\lambda)$ in the category $O$. 

Let $\HC_{\chi'}^1$ be the full subcategory of the category $\HC(U)$ consisting of all objects $M$  such that $M \m_{\chi'}=0$. Let $_\chi\HC_{\chi'}^1$ denote the full subcategory of  the category $\HC(U)$ consisting of all objects $M$ such that $\m_{\chi}^k M=0$ for some $k$ and $M\m_{\chi'}=0$.

One of the main results in \cite{BG80} is the classification of indecomposable projective objects in $\HC_{\chi'}^1$. To state the result, we need to introduce some notation. Let $\tilde{\Xi}:=\{(\mu,\lambda)\in \fh^*\x \fh^*|\mu-\lambda\in \Lambda\}$. The Weyl group $W$ acts diagonally on the set $\tilde{\Xi}$ under the dot action. Let $\Xi$  be the quotient set of $\tilde{\Xi}$ under this $W$-action. A pairing $(\mu, \lambda)$ is \emph{proper} if $\lambda$ is dominant and $\mu$ is minimal in the orbit $W_\lambda \cdot \mu$. Any element $\xi$ in $\Xi$ can always be represented by some proper pair $(\mu, \lambda)$.

Let $\Xi^r_{\chi'}$ be the subset of $\Xi$ consisting of all elements $\xi$ which  can be represented by a pair $(\mu, \lambda)$ for $\lambda \in \chi'$. Let $\Xi_{\chi, \chi'}$ be the  subset of $\Xi$ consisting of all elements $\xi$ which can be represented by a pair $(\mu, \lambda)$ such that $\mu \in \chi $ and $\lambda \in \chi'$. 
\begin{Prop}\label{prop indecom HCbim}
a) There is one-to-one correspondence between the indecomposable projective objects in $\HC^1_{\chi'}$ and the elements of the set $\Xi^r_{\chi'}$. Similarly, there is one to one correspondence between indecomposable projective objects in $_\chi\HC^1_{\chi'}$ and  the elements of the set $\Xi_{\chi, \chi'}$.

b) Let $P_\xi$ be the indecomposable projective object corresponding to $\xi$ in $\Xi$. If $(\mu, \lambda)$ is a proper pairing representing $\xi$, then $P_{\xi}\otimes_U \Delta(\lambda)\cong P(\mu)$; here $P(\mu)$ is the projective cover of the simple object $L(\mu)$. 

\end{Prop}
 We refer more details about the classification to Section $5.4$ in \cite{BG80}.
\begin{Lem} Suppose there is a pair of dominant weights $(\mu, \lambda)$ such that $\mu\in \chi, \lambda\in \chi'$ and $\mu-\lambda \in \Lambda$, then the number of elements in $\Xi_{\chi, \chi'}$ is equal to the number of double cosets $W_\mu\backslash \widehat{W}_{[\lambda]}/W_\lambda$. Otherwise, the set $\Xi_{\chi, \chi'}$ is empty.
\end{Lem}
\begin{proof} The number of elements in $\Xi_{\chi, \chi'}$ is equal to the number of $W_\lambda$-orbits in the set $\{w\cdot \mu|w\cdot \mu \in W\cdot \mu \cap \lambda+\Lambda\}$. The latter set is just $\widehat{W}_{[\lambda]} \cdot \mu$. So the number of elements in $\Xi_{\chi, \chi'}$ is equal to the number of double cosets $W_\lambda\backslash \widehat{W}_{[\lambda]}/W_\mu$, which is the same as the number of double cosets $W_\mu\backslash \widehat{W}_{[\lambda]}/W_\lambda$.
\end{proof}

Let $\lambda$ be a dominant weight in $\chi'$. For any $M$ in $\HC(U)$, the tensor product $M\otimes_{U} \Delta(\lambda)$ is contained in the category $O$. We define the functor: 
\[ T_\lambda: \HC(U) \rightarrow O\]
by $T_\lambda (M):= M \otimes_U \Delta(\lambda)$ for any $M$ in $\HC(U)$. Restricting to the full subcategory $\HC^1_{\chi'}$, we have the functor $T_\lambda: \HC^1_{\chi'}\rightarrow O$.

Denote $\SP_\lambda$ be the class of projective objects of the category $O$ consisting of direct sums of projective modules $P(\psi)$, where $\psi \in \lambda+\Lambda$ such that $\psi$ is minimal in the orbit $W_\lambda \cdot \psi$ and $P(\psi)$ is the projective cover of the simple object $L(\psi)$. An object $M$ in $O$ is called \emph{$\SP_\lambda$-representable} if there is an exact sequence $P_2\rightarrow P_1\rightarrow M\rightarrow 0$ with $P_1, P_2\in \SP_\lambda$.  Let $O_{\SP_\lambda}$ denote the full subcategory of the category $O$ consisting of all $\SP_\lambda$-representable objects. Let $O_{\lambda+\Lambda}$ be the full subcategory of the category $O$ consisting of all objects whose weights are contained in the lattice $\lambda+\Lambda$. We see that the category $O_{\SP_\lambda}$ is a full subcategory of the category $O_{\lambda+\Lambda}$. Moreover, if $\lambda$ is regular then $O_{\SP_\lambda}=O_{\lambda+\Lambda}$.
The following theorem is Theorem $5.9$ in \cite{BG80}:
\begin{Thm} \label{Thm BGG equivalence}For any dominant weight $\lambda \in \chi'$,  the functor $T_\lambda: \HC^1_{\chi'} \rightarrow 
O$ is a fully faithful.

a) If $\lambda$ is regular, then $T_\lambda$ is exact and induces an equivalence of categories $\HC^1_{\chi'} \cong O_{\lambda+\Lambda}$.

b) In general, $T_\lambda$ induces an equivalence of categories $\HC^1_{\chi'} \cong O_{\SP_\lambda}$. But $T_\lambda$ may not be exact since $O_{\SP_\lambda}$ is not closed under taking subquotients. 

\end{Thm}
\begin{Cor} \label{cor BGG equivalence}
Suppose $\chi$ and $\chi'$ are compatible with a pair of dominant weights $(\mu, \lambda)$ such that $\mu\in \chi, \lambda\in \chi'$ and $\mu-\lambda \in \Lambda$. Assume $\chi'$ is regular.  Then $T_\lambda$ induces an equivalence of categories $_\chi \HC^1_{\chi'} \cong \bigoplus_{\mu'}O_{\mu'}$, where $\mu'$ runs over all dominant weights in the orbit $\widehat{W}_{[\lambda]}\cdot\mu=W\cdot \mu \cap \lambda+\Lambda$.  \end{Cor}

%%%%%%%%%%%%%%%%Complete category O and HC bim
\section{Deformed Category $\hat{O}$ and Complete Harish-Chandra bimodules} \label{sec: Deform O and Complete HC}
\subsection{Deformed category $\hat{O}$} \label{ssec: Deformed Cat O}\

Let us recall $R=\C[\fh^*]^{\wedge_0}$, the completion of $\C[\fh^*]$ at $0$. Set $U_R(\g)=U(\g) \otimes R$. Let $z$ denote the natural map $z: S(\fh) \cong \C[\fh^*]\rightarrow R$. The deformed category $\hat{O}$ is the full subcategory of the category of $U_R(\g)$-modules consisting of all objects  $M$ satisfying:
\begin{enumerate}[label={(\bfseries T\arabic*)}]
\item \label{T1}Under the action of the Cartan subalgebra $\fh$, we have that $M=\bigoplus_{\lambda\in \fh^*} M_{\lambda}$, where $M_{\lambda}:=\{m \in M| hm=(\lambda(h)+z(h))m \;\text{for all $h\in \fh$}\}$
\item \label{T2} The action of $\n$ on $M$ is locally nilpotent.
\item \label{T3} $M$ is finitely generated over $U_R(\g)$.
\end{enumerate}

The category $\hat{O}$ shares many properties with the category $O$. Any $U_R(\g)$-submodule and $U_R(\g)$-quotient of any object in $\hat{O}$ is again contained in $\hat{O}$. For any $M$ in $\hat{O}$, the weight space $M_\lambda$ is finitely generated over $R$ for all $\lambda$ in $\fh^*$. The following properties of $\hat{O}$ can be founded in \cite{So92} or obtained by similar treatments for the category $O$ in Section $\ref{ssec: Cat O}$.

The deformed Verma module $\hat{\Delta}(\lambda)=U(\g) \otimes_{U(\fb)} (\C_\lambda \otimes R)$ defined as follows is contained in $\hat{O}$: The $U(\fh)$-action on the tensor product $\C_\lambda\otimes R$ makes $\C_\lambda\otimes R$ into a $(U(\fh), R)$-bimodule. Under the natural map $U(\fb)\rightarrow U(\fh)$, we can view $\C_\lambda\otimes R$ as a $(U(\fb), R)$-bimodule. Therefore, the induction $\hat{\Delta}(\lambda)$ is a $(U(\g), R)$-bimodule, equivalently, a $U_R(\g)$-module. It is then obvious that $\hat{\Delta}(\lambda)$ belongs to $\hat{O}$.

The center of $U_R(\g)$ is $Z\otimes R$. The closed points in Spec $(Z\otimes R)$ are of the form $(\chi, \m_0)$ in which $\chi$ is a point in $\fh^*/(W, \cdot)$ and $\m_0$ is the unique maximal ideal in $R$.  Any object in $\hat{O}$ has a finite filtration whose subquotients are quotients of deformed Verma modules. The support of the deformed Verma module $\hat{\Delta}(\lambda)$ contains the unique closed point $(\chi_\lambda, \m_0)$. For any dominant weight $\lambda$, let $\hat{O}_{\lambda}$ be the full subcategory of $\hat{O}$ consisting of objects whose supports in Spec $(Z\otimes R)$ contain the unique closed point $(\chi_\lambda, \m_0)$ and whose weights lie in the coset $\lambda +\Lambda_r$ in $\fh^*$. Then we have the decomposition of categories: $\hat{O}=\bigoplus_\lambda \hat{O}_\lambda$ in which $\lambda$ runs over all dominant weights in $\fh^*$. 
Let $\pr_\lambda: \hat{O}\rightarrow \hat{O}_\lambda$ be the natural projection. 

Let $\hat{O}_{\chi}=\bigoplus  \hat{O}_{\lambda}$ where $\lambda$ runs over all dominant weights in the orbits $\chi$. Then $\hat{O}_{\chi}$ is the full subcategory of $\hat{O}$ consisting of all objects whose  supports in Spec $(Z\otimes R)$ contain the unique closed point $(\chi, \m_0)$. There is a decomposition $\hat{O}=\bigoplus \hat{O}_{\chi}$. If $\lambda$ is integral and dominant then $\hat{O}_{\chi_\lambda}=\hat{O}_\lambda$. Let $\pr_{\chi}: \hat{O}\rightarrow \hat{O}_{\chi}$ be the natural projection. 

% We need to emphasize that it is crucial to consider the support of objects in $\hat{O}$ in Spec $(Z\otimes R)$ instead of Spec $Z$  since the support in Spec $Z$ does not allow to distinguish objects in $\hat{O}$. For example, the support of $\hat{\Delta}(\lambda)$ in Spec $Z$ is the whole Spec $Z$ for any $\lambda\in \fh^*$.

Let $\lambda$ be a dominant weight. Although the category $\hat{O}_\lambda$ is not Artinian, it does have finitely many simples $L(w\cdot \lambda)$ for all $w$ in $W_{[\lambda]}$, and it has enough projectives: $\hat{\Delta}(\lambda)$ is projective, and every simple is covered by some projective object of the form  $\pr_\lambda (V\otimes \hat{\Delta}(\lambda))$ for some finite dimensional $\g$-representation $V$. We can view the category $\hat{O}_\lambda$ as a deformation of the category $O_\lambda$.

For any two projective objects $\hat{P}_1, \hat{P}_2$ in $\hat{O}$, the space $\Hom_{U_R(\g)}(\hat{P}_1, \hat{P}_2)$ is a free $R$-module of finite rank,
 and for any $R$-algebra $R'$,  the canonical map:
 \Eq{\label{Hom of proj}\Hom_{U_R(\g)}(\hat{P}_1, \hat{P}_2) \otimes_R R' \rightarrow \Hom_{U_{R'}(\g)}(\hat{P}_1\otimes_R R', \hat{P}_2\otimes_R R')}
 is an isomorphism; see \cite[Theorem $5$]{So92}.

\begin{Lem} \label{prop finite homological O}
    Any object in $\hat{O}_\lambda$ has a finite resolution by projective objects.
\end{Lem}
\begin{proof} The category $\hat{O}_\lambda$ is a highest weight category over the finite cohomological dimensional ring $R$. Hence, the lemma follows from general results about highest weight categories, see \cite[Proposition $4.23$]{R07}.
\end{proof}

The following objects are variants of objects in Definition \ref{def translation functor}:
\begin{Def}
a) For any finite dimensional $\g$-representation $V$, the \emph{projective functor} $\CP_{\chi'\rightarrow \chi}^V:\hat{O}_{\chi'}\rightarrow \hat{O}_{\chi}$ is defined by $\pr_{\chi}(V\otimes M)$ for any $M$ in $\hat{O}_{\chi'}$.

b) Suppose there is a pair of dominant weights $(\mu,\lambda)$ such that $\mu\in \chi, \lambda\in \chi'$ and $ \mu-\lambda \in \Lambda$. The \emph{translation functor} $\CT_{\lambda \rightarrow \mu}: \hat{O}_{\chi'}\rightarrow \hat{O}_{\chi}$ is defined by $\CP_{\chi'\rightarrow \chi}^{L(\mu-\lambda)}$.
 
\end{Def}

The projective functor $\CP^V_{\chi'\rightarrow \chi}$ sends projective objects in $\hat{O}_{\chi'}$ to projective objects in $\hat{O}_{\chi}$. We also need the following result later. The proof is similar to the one of Lemma \ref{lem image of Verma}.
\begin{Lem} Suppose there is a pair of dominant weights $(\mu, \lambda)$ such that $\mu\in \chi, \lambda\in \chi'$ and $\mu-\lambda \in \Lambda$. Assume that $W_\lambda\subset W_\mu$. Then $\CT_{\lambda\rightarrow \mu}(\hat{\Delta}(w\cdot \lambda))=\hat{\Delta}(w\cdot \mu)$ for any $w\in W_{[\lambda]}=W_{[\mu]}.$
\end{Lem}

We now will focus on the relation between the $Z$-action and the $R$-action on any object of $\hat{O}_\lambda$. For  each $\lambda$ in $\fh^*$, we define the map $t_\lambda: \C[\fh^*]\rightarrow \C[\fh^*]$  by $(t_\lambda f)(\nu)=f(\nu+\lambda)$ for any $f\in \C[\fh^*], \nu\in \fh^*$. This map extends to an isomorphism of algebras $\C[\fh^*]^{\wedge_\lambda,(W_\lambda, \cdot)}\cong \C[\fh^*]^{\wedge_0, W_\lambda}$, here $W_\lambda$ is the stabilizer of $\lambda$ in $W$ under the dot action. The Harish-Chandra isomorphism $Z\cong \C[\fh^*]^{(W,\cdot)}$ gives an isomorphism $Z^{\wedge_{\chi'}}\cong \C[\fh^*]^{\wedge_\lambda, (W_\lambda, \cdot)}$. Hence, we can consider the following composition:
\Eq{\label{identify complete Z}\epsilon_\lambda: Z^{\wedge_{\chi_\lambda}}\xrightarrow[]{\cong} \C[\fh^*]^{\wedge_\lambda, (W_\lambda, \cdot)}\xrightarrow[]{\cong} \C[\fh^*]^{\wedge_0, W_\lambda}=R^{W_\lambda}.}

Now, we fix a dominant weight $\lambda$ and consider the category $\hat{O}_\lambda$. For each $M$ in $\hat{O}_\lambda$, the deformed weight space $M_\mu$ of $M$ is a finitely generated $R$-module, hence complete and seperated in the $\m_0$-adic topology. The action of $Z$ preserves all deformed weight spaces $M_\mu$.
\begin{Lem}a)  The action of $Z$ on the deformed weight space $M_\mu$ uniquely lifts to an action of $Z^{\wedge_{\chi_\lambda}}$. As a result, the action of $Z$ on $M$ can be uniquely extended to an action of $Z^{\wedge_{\chi_\lambda}}$ on $M$ by endomorphisms that are continuous in the $\m_0$-adic topology.

b) For any $w\in W_{[\lambda]}$, the action of $Z^{\wedge_{\chi_\lambda}}$ on $\hat{\Delta}(w\cdot \lambda)$ satisfies $zm =\epsilon_{w\cdot \lambda}(z)m$ for any $z\in Z^{\wedge_{\chi_\lambda}}$ and $ m \in \hat{\Delta}(w\cdot \lambda)$.

\end{Lem}
\begin{proof}
a) The action of $Z$ on $M_\mu$ equips $M_\mu$ with an $\m_{\chi_\lambda}$-adic topology. Since $M\otimes_R R/\m_0^k$ belongs to the category $O_\lambda$ for any integer $k$, the $\m_{\chi_\lambda}$-adic topology on $M_\mu$ is at least finer than the $\m_0$-adic topology on $M_\mu$. On the other hand, $M_\mu$ is complete and seperated in the $\m_0$-adic topology. Therefore, the action of $Z$ on $M_\mu$ can be uniquely extended to an action of $Z^{\wedge_{\chi_\lambda}}$ on $M_\mu$.

b)  The action map  $Z \rightarrow \End(\hat{\Delta}(w\cdot \lambda)) \cong R$ is just the following composition:
\[ Z \xrightarrow[]{\cong}\C[\fh^*]^{(W, \cdot)} \xrightarrow[]{t_{w\cdot \lambda}} \C[\fh^*] \rightarrow R.\]
Therefore, the claim follows by comparing with the construction  of $\epsilon_{w\cdot\lambda}$ in $(\ref{identify complete Z})$ with a notice  that $\chi_\lambda=\chi_{w\cdot \lambda}$.
\end{proof}
Via the identification $Z^{\wedge_{\chi_\lambda}} \cong R^{W_\lambda}$ in $(\ref{identify complete Z})$, we have left $R^{W_\lambda}$-actions on each object in $\hat{O}_\lambda$. Therefore, we have left $R^{W_{[\lambda]}}$-actions on each object in $\hat{O}_\lambda$. On the other hand, the right $R$-actions give right $R^{W_{[\lambda]}}$-actions on each object in $\hat{O}_\lambda$.
\begin{Cor} The left action  and the right action of $R^{W_{[\lambda]}}$ on each object in $\hat{O}_\lambda$ coincide.
\end{Cor}
\begin{proof} 
Since each object in $\hat{O}_\lambda$ can be covered by some object of the form $\pr_\lambda( V\otimes \hat{\Delta}(\lambda))$, it suffices to prove the statement for $M=\pr_\lambda(V\otimes \hat{\Delta}(\lambda))$. We know that $M$ has a finite filtration whose successive quotients isomorphic to $\hat{\Delta}(w\cdot \lambda)$ for some $w\in W_{[\lambda]}$, in particular, $M$ is a free $R$-module.

Let $Q$ be the fraction field of $R$. Let $O_Q$ denote the category $O$ over $U_Q(\g)$. The $U_Q(\g)$-module $M\otimes_R Q$ belongs to the category $O_Q$ and has a finite filtration whose successive quotients isomorphic to $\hat{\Delta}_Q(w\cdot \lambda):=\hat{\Delta}(w\cdot \lambda)\otimes_R Q$ for some $w\in W_{[\lambda]}$. Each Verma module $\hat{\Delta}(w\cdot \lambda)$ belongs to different blocks in the category $O_Q$. Therefore, $M\otimes_R Q$ is the direct sum of $U_Q(\g)$-modules of the form $\hat{\Delta}_Q(w\cdot \lambda)$, hence two $R^{W_{[\lambda]}}$-actions on $M\otimes_R Q$ coincide. Since $M$ is free over $R$, we have the inclusion map $M \hookrightarrow M\otimes_R Q$, therefore, two $R^{W_{[\lambda]}}$-actions on $M$ coincide.
\end{proof}

\begin{Rem} \label{rem3}We will also consider the analogs $\hat{O}^r, O^r$ consisting of right $U_R(\g)$-modules. All results in Section $\ref{ssec: Cat O}$ and $\ref{ssec: Deformed Cat O}$ can be carried to $\hat{O}^r$ and $O^r$ with suitable modifications.
\end{Rem}

\subsection{Complete Harish-Chandra bimodules} \label{ssec:  complete HCbim}\

%Let $\lambda, \mu$ be two dominant weights. Let us recall the algebras $U^\chi=U\otimes_Z Z^{\wedge_\chi}, U^{\chi'}=U\otimes_Z Z^{\wedge_{\chi'}}$. Let us consider the category $U^{\mu,\lambda}$-bim of $(U^\chi, U^{\chi'})$-bimodules.

Let $U^{\chi, \chi'}\bim$ denote  the category  of $(U^\chi, U^{\chi'})$-bimodules. 
\begin{Def}An object $M$ in $U^{\chi, \chi'}$-bim is Harish-Chandra if it satisfies the following axioms:

\begin{enumerate}[label={(\bfseries S\arabic*)}]
\itemsep-0.3em 
\item \label{S1} The adjoint $\g$-action on $M$ is locally finite.
\item  \label{S2}$M$ is finitely generated as a left $U^\chi$-module.
\end{enumerate}
\end{Def}
Let $\HC(U^{\chi, \chi'})$ denote the full subcategory of the category $U^{\chi, \chi'}$-bim consisting of all Harish-Chandra bimodules. One may wonder why we require $M$ to be finitely generated as a left $U^\chi$-module instead of being finitely generated as a bimodule as in \ref{R2}. This is because using the  axiom \ref{S2} makes some arguments easier. Nevertheless, we can replace the axiom \ref{S2} by an equivalent axiom as  in the following lemma:
\begin{Prop}\label{prop: equivalent axioms}
Modulo \ref{S1}, the axiom \ref{S2} is equivalent to the following axiom:
\begin{enumerate}[label={(\bfseries S\arabic*')}, start=2]
\item \label{S22} $M$ is finitely generated as a right $U^{\chi'}$-module.
\end{enumerate}
\end{Prop}
The proof of Proposition \ref{prop: equivalent axioms} will be provided after some preparation. Recall that $\CZ=Z\otimes Z$ is the center of $U \otimes U^{\textnormal{op}}$. Let $\CZ^{\chi, \chi'}$ be the completion of $\CZ$ at the closed point $(\chi, \chi')$ in $\fh^*/(W, \cdot) \x \fh^*/(W, \cdot)$. For any $M$ in $\HC(U)$, we will show that the tensor product $M \otimes_{\CZ} \CZ^{\chi,\chi'}$ is finitely generated as a left $U^\chi$-module, therefore,  $M\otimes_{\CZ}\CZ^{\chi, \chi'}$ belongs to $\HC(U^{\chi, \chi'})$. By Proposition \ref{prop property of HCbim}, for any finite dimensional $\g$-representation $V$, the space $\Hom_\g(V,M)$ is finitely generated as a left $Z$-module and as a right $Z$-module. Therefore, $\Hom_\g(V,M) \otimes_\CZ\CZ^{\chi, \chi'}$ is the completion of the $\CZ$-module $\Hom_\g(V,M)$ at the closed point $(\chi, \chi')$, and $Z^{\wedge_\chi}\otimes_Z \Hom_\g(V,M)$ is the completion of the left $Z$-module $\Hom_\g(V,M)$ at the ideal defining the closed subset $\chi\x \fh^*/(W, \cdot)$. Furthermore, by Remark \ref{rem support of HCbim}, there is a set $S$ of finitely many points in $\fh^*/(W, \cdot)$ such that 
\[\textnormal{Supp}\Hom_\g(V, M)\cap \chi\x \fh^*/(W, \cdot)=\chi\x S.\]
Let $I$ be the defining ideal of $\chi \x S$ in $\CZ$, and let $\CZ^{\wedge_I}$ be the completion of $\CZ$ at the ideal $I$. Then we have  isomorphisms of $Z^{\wedge_\chi}\otimes Z$-modules:
\[ Z^{\wedge_\chi}\otimes_Z \Hom_\g(V,M)\cong  \Hom_\g(V,M)\otimes_\CZ \CZ^{\wedge_I} \cong \bigoplus_{\chi'\in S}\Hom_\g(V,M)\otimes_\CZ \CZ^{\chi, \chi'}.\]
Similarly, we also have a decomposition of $Z\otimes Z^{\wedge_{\chi'}}$-modules:
\[ \Hom_\g(V,M)\otimes_Z Z^{\wedge_{\chi'}} \cong \bigoplus_\chi \Hom_\g(V,M)\otimes_\CZ\CZ^{\chi, \chi'}.\]
These two decompositions imply the following lemma:
\begin{Lem} For any $M$ in $\HC(U)$, we have the following decompositions:
\Eqn{ Z^{\wedge_\chi}\otimes_Z M\cong \bigoplus_{\chi'} M \otimes_\CZ\CZ^{\chi, \chi'} \tab \text{of $(U^\chi,U)$-bimodules,}\\
M\otimes_Z Z^{\wedge_{\chi'}} \cong \bigoplus_\chi M\otimes_\CZ\CZ^{\chi, \chi'} \tab \text{of $(U,U^{\chi'})$-bimodules.}}
As a consequence, $M\otimes_\CZ\CZ^{\chi, \chi'}$ is finitely generated as a left $U^\chi$-module and as a right $U^{\chi'}$-module. 
\end{Lem}
Now we have that $M\otimes_\CZ \CZ^{\chi, \chi'}$ is an object in $\HC(U^{\chi, \chi'})$.
\begin{Def} Define the functor $C^{\chi, \chi'}: \HC(U)\rightarrow \HC(U^{\chi, \chi'})$ by 
\[ C^{\chi, \chi'}(M):=M\otimes_\CZ\CZ^{\chi, \chi'}.\]
\end{Def}
The functor $C^{\chi, \chi'}$ is exact. For any finite dimensional $\g$-representation $V$, we set 
\Eq{\label{Def of proj bimod} P^{\chi,\chi'}_V:=C^{\chi, \chi'}(V\otimes U(g)).}
\begin{Def} \label{def translation bimod}Suppose there is a pair of dominant weights $(\mu, \lambda)$ such that $\mu \in \chi, \lambda \in \chi'$ and $\mu-\lambda \in \Lambda$.  The \emph{translation bimodule} $P^{\mu, \lambda}$ is defined by $P^{\mu, \lambda}:=C^{\chi, \chi'}(L(\mu-\lambda)\otimes U(\g))$.
\end{Def}

As its name suggests, the translation bimodules are closely related to the translation functors of the category $\hat{O}$, see Lemma \ref{lem image of complete Verma}. We now can give the proof for Proposition \ref{prop: equivalent axioms}.

\begin{proof}[Proof of Proposition \ref{prop: equivalent axioms}]
Let $M$ be a $(U^\chi, U^{\chi'})$-bimodule which satisfies the axiom \ref{S1} and \ref{S2}. We will show that $M$ satisfies the axiom \ref{S22}. There is a finite dimensional $\g$-representation $V$ such that we have a surjective morphism of $(U^\chi, U)$-bimodules:
\[ U^\chi\otimes V\cong \oplus_\lambda P^{\chi, \chi'}_V \twoheadrightarrow  M\]
which then induces an epimorphism of $(U^\chi, U^{\chi'})$-bimodules:
\[ P^{\chi,\chi'}_V\twoheadrightarrow M.\]
Since $U^{\chi'}$ is Noetherian  and $V\otimes U^{\chi'}=\oplus_\chi P^{\chi, \chi'}_V$, it follows that $P^{\chi, \chi'}_V$ is finitely generated as a right $U^{\chi'}$-module. Therefore, $M$ is finitely generated as a right $U^{\chi'}$-module. 

The proof for the other direction is similar.
\end{proof}

We now list various basis properties of the category $\HC(U^{\chi, \chi'})$.
\begin{Cor} \label{prop of HCchi}
a) $\Hom_{\HC(U^{\chi, \chi'})}(P^{\chi, \chi'}_V, M)=\Hom_\g(V,M^{\ad})$ for any $M$ in $\HC(U^{\chi, \chi'})$. Hence, $P^{\chi, \chi'}_V$ is projective in $\HC(U^{\chi, \chi'})$.

b) $P^{\chi, \chi'}_V$ is projective as a left $U^\chi$-module and projective as a right $U^{\chi'}$-module.

c)  Any object in $\HC(U^{\chi, \chi'})$ is covered by some projective object of the form $P^{\chi, \chi'}_V$. In particular, any projective object in $\HC(U^{\chi, \chi'})$ is the direct summand of $P^{\chi, \chi'}_V$ for some $V$.

d) $\Hom_\g(V,M)$ is finitely generated as a left $Z^{\wedge_\chi}$-module, as a right $Z^{\wedge_{\chi'}}$-module, and as a $\CZ^{\chi, \chi'}$-module.

e) Let $J_{\chi, \chi'}$ be the maximal ideal of $\CZ$ corresponding to the closed point $(\chi, \chi')$. If $M/MJ_{\chi, \chi'}=0,$ then $M=0$.

f) $\Hom_\g(V,M)=\varprojlim \Hom_\g(V,M/\m_{\chi}^kM)=\varprojlim \Hom_\g(V,M/M\m_{\chi'}^k)$ for any finite dimensional $\g$-representaion $V$.

g) \label{cor: surj maps} Let $M, N$ be objects in $\HC(U^{\chi, \chi'})$. Let us assume that a map $\phi: M \rightarrow N$ induces a surjective map $M/M\m_{\chi'}\rightarrow N/N\m_{\chi'}$. Then $\phi$ is also surjective.

\end{Cor}
\begin{proof}
a) It is obvious from the definition of $P^{\chi, \chi'}_V$.

b) It follows by the direct sum decompositions $U^\chi\otimes V=\oplus_{\chi'} P^{\chi,\chi'}_V$ and $V\otimes U^{\chi'}=\oplus_\chi P^{\chi, \chi'}_V$.

c) This has been showed in the proof of Proposition $\ref{prop: equivalent axioms}$

d)  By part $c)$, it is enough to prove the statement for $P^{\chi, \chi'}_V$, which is obvious.

e) Note that $\Hom_\g(V,M/MJ_{\chi, \chi'})\cong \Hom_\g(V,M)/J_{\chi, \chi'}\Hom_\g(V,M)$ for any finite dimensional $\g$-representation $V$. If $M/MJ_{\chi, \chi'}=0$ then $\Hom_\g(V,M)/J_{\chi, \chi'}\Hom_\g(V, M)=0$. On the other hand, $\Hom_\g(V,M)$ is finitely generated over $\CZ^{\chi, \chi'}$. Therefore, $\Hom_\g(V,M)=0$ for any finite dimensional $\g$-representation $V$ which implies $M=0$.

f) Note that  $\Hom_\g(V,M/\m_{\chi}^k M)=\Hom_\g(V,M)/\m_{\chi}^k\Hom_\g(V,M)$ for any $k$. Since $\Hom_\g(V,M)$ is finitely generated as a left $Z^{\wedge_\chi}$-module, we must have
\[ \Hom_\g(V,M)=\varprojlim \Hom_\g(V,M)/\m_{\chi}^k\Hom_\g(V,M)=\varprojlim \Hom_\g(V,M/\m_{\chi}^k M).\]

g) It is enough to show that the induced map 
\Eq{\label{eq10} \phi: \Hom_\g(V, M)\rightarrow \Hom_\g(V,N)}
is surjective for any finite dimensional $\g$-representation $V$. Since the induced map $M/M\m_{\chi'}\rightarrow N/N\m_{\chi'}$ is surjective, the following map is also surjective
\[ \Hom_\g(V, M)/\Hom_\g(V,M)\m_{\chi'}\rightarrow \Hom_\g(V,N)/\Hom_\g(V,N)\m_{\chi'}.\]
On the other hand, $\Hom_\g(V,M)$ and $\Hom_\g(V,N)$ are finitely generated as right $Z^{\wedge_{\chi'}}$-modules and $Z^{\wedge_{\chi'}}$ is a local ring. Therefore the map $(\ref{eq10})$ is surjective.
\end{proof}

\begin{Rem} \label{rem HCbim is zero} Since any object in $\HC(U^{\chi, \chi'})$ is covered by some projective object of the form $P^{\chi, \chi'}_V$, the category $\HC(U^{\chi, \chi'})$ is zero unless $\chi$ and $\chi'$ are compatible.
\end{Rem}

\begin{Lem} \label{lem generator V}There is a finite dimensional $\g$-representation $V_0$ such that $\Hom_\g(V_0, M)=0$ implies $M=0$ for any $M$ in $\HC(U^{\chi, \chi'})$.
\end{Lem}
\begin{proof}
Let $J_{\chi, \chi'}$ be the maximal ideal of $\CZ$ corresponding to the closed point $(\chi, \chi')$. In  Proposition \ref{prop char of HCbim}, let $V_0$ be a finite dimensional $\g$-representation such that $\Hom_\g(V,L)\neq 0$ for all simple Harish-Chandra bimodules $L$ such that $J_{\chi, \chi'} L=0$. 

Let $M$ in $\HC(U^{\chi, \chi'})$ such that $\Hom_\g(V_0, M)=0$. For any finite dimensional $\g$-representation $V$, we have $\Hom_\g(V,M)/J_{\chi, \chi'}\cong \Hom_\g(V, M/MJ_{\chi, \chi'})$. So we have $\Hom_\g(V_0, M/MJ_{\chi, \chi'})=0$, hence $M/MJ_{\chi, \chi'}=0$. Therefore, $M=0$.
\end{proof}

We now provide a characterization of Harish-Chandra bimodules in $\HC(U^{\chi, \chi'})$.
\begin{Prop} Let $M$ be a $(U^\chi, U^{\chi'})$-bimodule whose adjoint $\g$-action is locally finite. If $\Hom_\g(V,M)$ is finitely generated  as a left $Z^{\wedge_\chi}$-module for any finite dimensional $\g$-representation $V$ then $M$ is Harish-Chandra. Similarly, if $\Hom_\g(V,M)$ is finitely generated  as a right $Z^{\wedge_{\chi'}}$-module for any finite dimensional $\g$-representation $V$, then $M$ is also Harish-Chandra.
\end{Prop}

\begin{proof} Let $U^{\chi, \chi'}\bim^{\text{lf}}$ denote the full subcategory of the category $U^{\chi, \chi'}$-bim consisting of all objects whose adjoint $\g$-actions are locally finite. Any object in $U^{\chi, \chi'}\bim^{\text{lf}}$ is equal to the union of its Harish-Chandra subbimodules. By Lemma \ref{lem generator V}, there is a finite dimensional $\g$-representation $V_0$ such that $\Hom_\g(V_0,M)=0$ implies $M=0$ for any $M$ in $\HC(U^{\chi, \chi'})$. This implies that there is a finite dimensional $\g$-representation $V_0$ such that $\Hom_\g(V_0,M)=0$ implies $M=0$ for any $M$ in  $U^{\chi, \chi'}\bim^{\text{lf}}$.

Let us  prove the first statement of the proposition only since the proof for the second statement is similar. Suppose $M$ is in $U^{\chi, \chi'}\bim^{\text{lf}}$ such that $\Hom_\g(V,M)$ is finitely generated as a left $Z^{\wedge_\chi}$-module for any finitely dimensional $\g$-representation of $V$. In particular, the space $\Hom_\g(V_0, M)$ is finitely generated as a left $Z^{\wedge_\chi}$-module. One then can easily show that $M$ is Noetherian in the category $U^{\chi, \chi'}\bim^{\text{lf}}$.

% Let $M_1\subset M_2\subset \dots $ be an ascending chain of subbimodules of $M$ in $U^{\chi, \chi'}\bim^{\text{lf}}$ then $\Hom_\g(V_0,M_1)\subset \Hom_\g(V_0,M_2)\subset \dots $ is an ascending chain of left $Z^{\wedge_\chi}$-submodules of $\Hom_\g(V_0,M)$. Since $\Hom_\g(V_0,M)$ is finitely generated over $Z^{\wedge_\chi}$, it is Noetherian. The ascending chain of $\Hom_\g(V_0,M_i)$ terminates. Combining with the statement at the end of the previous paragraph, the ascending chain of $M_i$ also terminates. We see that $M$ is Noetherian in the category $U^{\chi, \chi'}\bim^{\text{lf}}$.

Note that the sum of two Harish-Chandra subbimodules in $M$ is also a Harish-Chandra subbimodule. Therefore, the fact that $M$ is Noetherian and equal to the union of its Harish-Chandra subbimodules implies that $M$ is Harish-Chandra. 
\end{proof}

 \subsection{Monoidal structure}\

 Let $M$ in $\HC(U^{\chi'', \chi})$ and $N$ in $\HC(U^{\chi, \chi'})$. Then $M\otimes_{U^{\chi}}N$ is an object in $\HC(U^{\chi'', \chi'})$ by Lemma \ref{prop: equivalent axioms}. Therefore, we can define functors:
 \Eqn{(-)\otimes_{U^\chi}(-): \HC(U^{\chi'',\chi})\x \HC(U^{\chi, \chi'}) \rightarrow \HC(U^{\chi'', \chi'})}
 by $(M,N)\mapsto M\otimes_{U^\chi} N$. These functors satisfy the unit and associativity axioms:
 \begin{itemize}
     \item In case $\chi''=\chi$, we have a canonical isomorphism of functors:
     \[ U^\chi \otimes_{U^\chi}(-)\cong \text{id},\]
     and in the case $\chi=\chi'$, we have a canonical isomorphism of functors:
     \[ (-)\otimes_{U^\chi} U^\chi \cong \text{id};\]
     \item For four points $\chi_1, \chi_2, \chi_3, \chi_4$ in $\fh^*/(W,\cdot)$, we have a canonical isomorphism:
     \[ ((-)\otimes_{U^{\chi_2}} (-))\otimes_{U^{\chi_3}}(-)\xrightarrow[]{\cong} (-)\otimes_{U^{\chi_2}}((-)\otimes_{U^{\chi_3}}(-))\]
     of functors from $\HC(U^{\chi_1, \chi_2})\x \HC(U^{\chi_2, \chi_3})\x \HC(U^{\chi_3, \chi_4})$ to $\HC(U^{\chi_1, \chi_4})$.
 \end{itemize}

 In particular, this construction equips $\HC(U^{\chi, \chi})$ with structure of a monoidal category. Of course, one needs to check some coherences, but it is not hard to see that they hold in this case.

\subsection{Relation between the category $\hat{O}$ and the complete Harish-Chandra bimodules}\

Let us define a functor
\[ \HC(U^{\chi, \chi'})\x \hat{O}_{\chi'}\rightarrow \hat{O}_{\chi}\]
by mapping $(M, Y)$ to $M\otimes_{U^{\chi'}}Y$ for any $M\in \HC(U^{\chi, \chi'})$ and any $Y\in \hat{O}_{\chi'}$.

In order to this functor to be well-defined, we need to check that $M\otimes_{U^{\chi'}}Y$ belongs to $\hat{O}_{\chi}$. Indeed, the tensor product $M\otimes_{U^{\chi'}} Y$ is naturally equipped with a finitely generated $U_R(\g)$-module structure come from the left $U(\g)$-action $M$ and the right $R$-action on $Y$ so that $M\otimes_{U^{\chi'}}Y$ satisfies the axiom \ref{T3}. The locally finiteness of the adjoint $\g$-action on $M$ implies that $M\otimes_{U^{\chi'}} Y$ has a deformed weight decomposition  as in the  axiom \ref{T1}. The adjoint $\n$-action on $M$ is locally nilpotent hence $M\otimes_{U^{\chi'}} Y$ satisfies the axiom \ref{T2}. So $M\otimes_{U^{\chi'}}Y$ is indeed an object in $\hat{O}$. Now, since the left $Z$-action on $M$ factors through the $Z^{\wedge_\chi}$-action, the support of $M\otimes_{U^{\chi'}} Y$ in Spec $Z\otimes R$ contains a unique closed point $(\chi, \m_0)$. Therefore, $M\otimes_{U^{\chi'}} Y$ is contained in $\hat{O}_{\chi}$.

Our main results in this section are Proposition \ref{prop actions on tranbimod} and Proposition \ref{prop number of complete HCbim}. 
\begin{Lem}\label{lem HCproj to Oproj} If $P$ is a projective object in $\HC(U^{\chi, \chi'})$ and $Y$  is a projective object in $\hat{O}_{\chi'}$ then $P \otimes_{U^{\chi'}} Y$ is a projective object in $\hat{O}_{\chi}$.
\end{Lem}
\begin{proof}
Since each projective  object in $\HC(U^{\chi, \chi'})$ is the direct summand of some $P^{\chi, \chi'}_V$, it is enough to prove the lemma in the case $P=P^{\chi, \chi'}_V=(V\otimes U(\g))\otimes_{\CZ}\CZ^{\chi, \chi'}$ . But then the lemma follows since $P^{\chi, \chi'}_V\otimes_{U^{\chi'}}Y \cong \pr_{\chi}(V\otimes Y)$ for any $Y$ in $\hat{O}_{\chi'}$.
\end{proof}

 The next proposition allows us to translate some statements about projective objects in the category $\HC(U^{\chi, \chi'})$ to statements about projective objects in the category $\hat{O}$. Furthermore, we will use this proposition to prove the Bernstein-Gelfand type equivalence between $\HC(U^{\chi, \chi'})$ and $\hat{O}_{\chi'}$ when $\chi'$ is regular, see Proposition \ref{prop number of complete HCbim}.
\begin{Prop}\label{prop ff on HCproj}Let $\lambda$ be a dominant weight in $\chi'$.  Let $P_1, P_2$ be two projective objects in $\HC(U^{\chi, \chi'})$. The following map
\[\Hom_{\HC(U^{\chi, \chi'})}(P_1,P_2) \otimes_{Z^{\wedge_{\chi'}}} R\xrightarrow[]{\psi}\Hom_{\hat{O}}(P_1\otimes_{U^{\chi'}}\hat{\Delta}(\lambda), P_2\otimes_{U^{\chi'}}\hat{\Delta}(\lambda))\]
is bijective, here we view $R$ as a $Z^{\wedge_{\chi'}}$-module via $Z^{\wedge_{\chi'}}\xrightarrow[]{\epsilon_\lambda} R^{W_\lambda}\hookrightarrow R$.
\end{Prop}
\begin{proof}
It suffices to prove the statement in the case $P_1=P^{\chi, \chi'}_V, P_2=P^{\chi, \chi'}_{V'}$ for some finite dimensional $\g$-representations $V,V'$. We set
\[ \uP_V= \; _\chi(V\otimes (U(\g)/\m_{\chi'})),\]
so that  $\uP_V$ is the direct summand of the Harish-Chandra bimodule $V\otimes (U(\g)/\m_{\chi'})$ whose support in Spec $\CZ$ is the closed point $(\chi, \chi')$. One can see that $\uP_V\cong P^{\chi, \chi'}_V/P^{\chi, \chi'}_V\m_{\chi'}$.

We claim that by specializing both sides of $\psi$ at the maximal ideal $\m_0$ of $R$, we get
\[ \Hom_{\HC(U)}(\uP_V, \uP_{V'})\xrightarrow[]{\upsi} \Hom_O(\uP_V\otimes_{U}\Delta(\lambda), \uP_{V'}\otimes_{U} \Delta(\lambda)).\]

It is easy to see that $P^{\chi, \chi'}_V\otimes_{U^{\chi'}} \hat{\Delta}(\lambda)/\m_0\cong \uP_V\otimes_{U}\Delta(\lambda)$. By Lemma \ref{lem HCproj to Oproj}, both $P^{\chi, \chi'}_V\otimes_{U^{\chi'}} \hat{\Delta}(\lambda)$ and $P^{\chi, \chi'}_{V'}\otimes_{U^{\chi'}} \hat{\Delta}(\lambda)$ are projective objects in $\hat{O}$. Therefore, by \eqref{Hom of proj},  the target of $\psi$ is a free $R$-module of finite rank and its  specilization the maximal ideal $\m_0$ is isomorphic to the target of $\upsi$ .

On the other hand, 
\Eqn{\Hom_{\HC(U^{\chi, \chi'})}(P^{\chi, \chi'}_V, P^{\chi, \chi'}_{V'})&\cong \Hom_\g(V, P^{\chi, \chi'}_{V'})\\
\Hom_{\HC(U)}(\uP_V, \uP_{V'})&\cong \Hom_\g(V,\uP_{V'}).}
Note that $\uP_{V'}\cong P^{\chi, \chi'}_{V'}/P^{\chi, \chi'}_{V'}\m_{\chi'}$. Therefore, the source of $\psi$ is finitely generated as a right $R$-module and its specialization at the maximal ideal $\m_0$ is isomorphic to the source of $\upsi$. So  our claim follows.

By Theorem \ref{Thm BGG equivalence}, the map $\upsi$ is bijective. Since the source of $\psi$ is finitely generated as a $R$-module and the target of $\psi$ is a free $R$-module of finite rank, the map $\psi$ must be an isomorphism.
\end{proof}

%\textcolor{red}{Lemma $\ref{cor: surj maps1}$ is incorrect. Other fact, If $\mu'=w\cdot \mu$ dominant then $W_{\mu'}=w W_\mu w^{-1}$. We see that $W_\lambda$ is not always contained in $W_{\mu'}$ unless $\lambda$ is regular and $W_\lambda$ is trivial. If $\lambda$ is regular then $P^{\mu, \lambda}$ is the direct sum of $W/W_{[\lambda]}$ bimodules parametrized by dominant weights $\mu' \cong \mu$. Let denote such direct sum by $P^{\mu', \mu, \lambda}$, we have $P^{\mu', \mu,\lambda}_{\dag}\cong Z^{\wedge_{\chi'}}$ but the left $Z^{\wedge_\chi}$-action defined by $ Z^{\wedge_{\mu'}}\xrightarrow[]{\epsilon_{\mu'}} R^{W_{\mu'}} \hookrightarrow R \xrightarrow[]{\epsilon_\lambda} Z^{\wedge_{\chi'}}$.}  

Suppose there is a pair of dominant weights $(\mu, \lambda)$ such that $\mu \in \chi, \lambda\in \chi'$ and $\mu-\lambda \in \Lambda$. We recall the translation bimodules $P^{\mu, \lambda} \in \HC(U^{\chi, \chi'})$ and $P^{\lambda, \mu}\in \HC(U^{\chi', \chi})$ in Definition \ref{def translation bimod}.

\begin{Lem}\label{lem image of complete Verma} The functor $P^{\mu, \lambda} \otimes_{U^{\chi'}} \bullet: \hat{O}_{\chi'} \rightarrow \hat{O}_{\chi}$ defined by $Y\mapsto P^{\mu, \lambda}\otimes_{U^{\chi'}} Y$ for any $Y$ in $\hat{O}_{\chi'}$ is isomorphic to the translation functor $\CT_{\lambda\rightarrow \mu}$. As a consequence, if $W_\lambda\subset W_\mu$, then $P^{\mu, \lambda}\otimes_{U^{\chi'}} \hat{\Delta}(w\cdot \lambda)\cong \hat{\Delta}(w\cdot \mu)$ for any $w\in W_{[\lambda]}$.
\end{Lem}
\begin{proof}This is obvious.
\end{proof}
 Let assume $W_\lambda\subset W_\mu$, then we have an inclusion:
 \Eq{\label{Z hom}Z^{\wedge_\chi}\xrightarrow[]{\epsilon_\mu} R^{W_\mu} \hookrightarrow R^{W_\lambda} \xrightarrow[]{\epsilon^{-1}_\lambda} Z^{\wedge_{\chi'}}.}
Via \eqref{Z hom}, we have a right $Z^{\wedge_\chi}$-action on $P^{\mu, \lambda}$ by the restriction of the right $Z^{\wedge_{\chi'}}$-action. Similarly, we have the left $Z^{\wedge_\chi}$-action on $P^{\lambda, \mu}$ by the restriction of the left $Z^{\wedge_{\chi'}}$-action. 

\begin{Prop} \label{prop actions on tranbimod}
Suppose that $ W_\lambda\subset W_\mu$. Then the left and right  $Z^{\wedge_\chi}$-actions on $P^{\mu, \lambda}$ coincide.  Similarly, the left and right $Z^{\wedge_\chi}$-actions on $P^{\lambda, \mu}$ coincide. 
\end{Prop}
\begin{proof}
Let us prove the statement for $P^{\mu, \lambda}$. By Lemma \ref{lem image of complete Verma}, we have the isomorphism $P^{\mu, \lambda}\otimes_{U^{\chi'}} \hat{\Delta}(\lambda) \cong \hat{\Delta}(\mu)$. So by  Proposition \ref{prop ff on HCproj}, we have an inclusion
\Eq{\label{eq4} \text{End}_{\HC(U^{\chi, \chi'})}(P^{\mu, \lambda})\hookrightarrow \End_{\hat{O}}(\hat{\Delta}(\mu)).}
The left and right actions of $Z^{\wedge_\chi}$ on $P^{\mu,\lambda}$ induce the same actions on  $\hat{\Delta}(\mu)$, which coincide with the left action of $Z^{\wedge_\chi}$ on $\hat{\Delta}(\mu)$. The inclusion $(\ref{eq4})$ then implies that the left and right actions of $Z^{\wedge_\chi}$ on $P^{\mu, \lambda}$ coincide. The proof of the statement for  $P^{\lambda, \mu}$ is similar if we  consider the category $\hat{O}^r$ in the place of the category $\hat{O}$.
\end{proof}

\begin{Prop} \label{prop number of complete HCbim}Suppose there is a pair of dominant weights $(\mu, \lambda)$ such that $\mu\in \chi, \lambda\in \chi'$ and $\mu-\lambda \in \Lambda$.

a) There is one to one correspondence between the indecomposable projective objects in $\HC(U^{\chi, \chi'})$ and the elements of the set $\Xi_{\chi, \chi'}$. In particular, the number of indecomposable projective objects in $\HC(U^{\chi, \chi'})$ is equal to the number of double cosets $W_\mu\backslash \widehat{W}_{[\lambda]}/W_\lambda$.

b) Define the  functor $T_\lambda: \HC(U^{\chi, \chi'})\rightarrow \hat{O}_{\chi}$ by $T_\lambda(M)=M\otimes_{U^{\chi'}} \hat{\Delta}(\lambda)$ for any $M$ in $\HC(U^{\chi, \chi'})$. If $\lambda$ is regular,  then $T_\lambda$ defines an equivalence of categories $\HC(U^{\chi, \chi'})\cong \oplus_{\mu'}\hat{O}_{\mu'}$, where $\mu'$ runs over all dominant weights in the orbit $\widehat{W}_{[\lambda]}\cdot \mu =W\cdot \mu \cap \lambda+\Lambda$.

\end{Prop}
\begin{proof}

a) Let $V$ be a finite dimensional $\g$-representation. Then $\text{End}_{\HC(U^{\chi, \chi'})}(P^{\chi,\chi'}_V)$ is an algebra which is a free module of finite rank over $Z^{\wedge_{\chi'}}$. Hence, $\text{End}_{\HC(U^{\chi, \chi'})}(P^{\chi, \chi'}_V)$ is complete in the $\m_{\chi'}$-adic topology. On the other hand,
    \[ \End_{\HC(U^{\chi, \chi'})}(P^{\chi, \chi'}_V)/\m_{\chi'} \cong \End_{_\chi\HC_{\chi'}^1}(P^{\chi, \chi'}_V/P^{\chi, \chi'}_V\m_{\chi'}).\]
    Therefore, by the standard argument of lifting idempotents,  we see that there is a one to one correspondence between the direct summands of $P^{\chi, \chi'}_V$ in the category $\HC(U^{\chi, \chi'})$ and the direct summands of $P^{\chi, \chi'}_V/P^{\chi, \chi'}_V \m_{\chi'}$  in the category $_\chi\HC_{\chi'}^1$. On the other hand, any object in $\HC(U^{\chi, \chi'})$ can be covered by some object of the form $P^{\chi, \chi'}_V$, while any object in $_\chi\HC_{\chi'}^1$ can be covered by some object of the form $P^{\chi, \chi'}_V/P^{\chi, \chi'}_V \m_{\chi'}$. Therefore, part $a)$ follows by  Proposition \ref{prop indecom HCbim}.

b) This part follows by Proposition \ref{prop ff on HCproj} and the categorical result in \cite[Proposition $5.10$]{BG80}.
\end{proof}

\section{The functor $\bullet_\dag$}\label{sec: restriction functor}
\subsection{The functor $\bullet_\dag: \HC(U)\rightarrow Z\bim^{C(G)}$} \label{ssec: restriction functor}\

Let $\CN, \BO_{\reg}$ denote the nilpotent cone and the regular nilpotent orbit in $\g$, respectively. Let $\partial \CN$ denote the complement of $\BO_{\reg}$ in $\CN$. Under the Killing form on $\g$, we can identify $\g$ with $\g^*$ so that we can view $\CN, \BO_{\reg}$ and $\partial \CN$  as subvarieties of $\g^*$. Let $e$ be a nilpotent element in $\g$. Let $\xi$ denote the image of the regular nilpotent element $e$ in $\g^*$ under this identification. 

Let us recall the PBW filtration $\{F_iU\}_{i\geq0}$ of $U(\g)$. Recall that a filtration $\{F_iM\}_{i\geq 0}$ of an object $M$ in $\HC(U)$ is good if $[F_iU, F_jM]\subset F_{i+j-1}M$ for any $i,j \geq 0$, and  the associated graded object $\gr _F M$ is a finitely generated module over $\C[\g^*]$. The \emph{associated variety} $V(M)$ of $M$ is the support of $\gr _F M$ in $\g^*$ for some good filtration $F$ of $M$. This associated variety does not depend on the choice of a good filtration on $M$.  Let $\HC_\CN(U)$ denote the full subcategory of the category $\HC(U)$ consisting of all objects $M$ such that $V(M)=\CN$. The definition of $\HC_{\partial \CN}(U)$ is similar. The multiplicity of $\gr _F M$ on the nilpotent cone $\CN$ for some good filtration $F$ is denoted by $\textnormal{mult}_\CN M$. The number $\textnormal{mult}_\CN M$ does not depend on the choice of good filtration on $M$.

Let $Z\bim^{C(G)}$ denote  the category of $\mathfrak{X}$-graded $Z$-bimodules. Let $Z\text{-bimf}^{C(G)}$ denote the full subcategory of the category $Z\bim^{C(G)}$ consisting of all objects which are finite dimensional vector spaces over $\C$. In \cite{IL11}, for the regular nilpotent element $e$, we have the following two functors:
\Eq{\label{Restriction functor} \bullet_\dag: \HC(U) \rightarrow Z\bim^{C(G)},\tab \bullet^\dag: Z\textnormal{-bimf}^{C(G)}\rightarrow \HC_\CN(U);}
here $C(G)$ is the center of the simply connected algebraic group $G$.

Let us briefly describe the construction of $\bullet_\dag$ in \cite{IL11}. We first form an $\mathfrak{sl}_2$-triple $(e,f,h)$. The element $h$ in the $\mathfrak{sl}_2$-triple $(e, h,f)$ defines an $\C^{\x}$-action $\gamma$ on $\g^*$. A new $\C^{\x}$-action is defined  by $t\cdot x=t^{-2}\gamma(t)x$ for any $x$ in $\g^*$, so  this new  $\C^{\x}$-action fixes $\xi$. The Slodowy slice $S$ corresponding to $\xi$ is just the Kostant slice in our case when $e$ is regular.  Let $\CU_\hbar$ denote the Rees algebra of $U(\g)$ for the PBW filtration. We can view $\xi $ as a homomorphism $\CU_\hbar \rightarrow \C$ and form the completion $\CU_\hbar^{\wedge_\xi}$ at the kernel of $\xi$. 

Set $V:=T_\xi \BO$; this is a symplectic space with the Kirillov-Kostant form $\omega$. It is $\C^{\x}$-stable under the new $\C^{\x}$-action. Consider its homogenized Weyl algebra $\CA_\hbar=T(V)[\hbar]/([u,v]-\hbar\omega(u,v))$. Thanks to the decompostion $T_\xi\g^*= V\oplus T_\xi S$, we can view $V$ as a subspace of $\g=(T_\xi \g^*)^*$. Then we can $\C^{\x}$-equivariantly lift the embedding $V\rightarrow \g $ to an embedding $V\rightarrow \CU_\hbar$ so that it lifts to an algebra homomorphism \cite[$\mathsection 2.1$]{IL15}
\[\CA_\hbar^{\wedge_0}\hookrightarrow \CU_\hbar^{\wedge_\xi};\]
here $\bullet^{\wedge_0}$ stands for the completion at the maximal ideal of $0$. Let $\CW'_\hbar$ be the centralizer of the image of $V$ in $\CU_\hbar^{\wedge_\xi}$. Then we have the decomposition \cite[$\mathsection 2.3$]{IL11}
\[ \CU_\hbar^{\wedge_\xi}\cong \CA_\hbar^{\wedge_0}\widehat{\otimes}_{\C[[\hbar]]} \CW'_\hbar,\]
where $\widehat{\otimes}$ is the completed tensor product.

The algebra $\CW'_\hbar$ comes with a $\C^{\x}$-action. Let $\CW_\hbar$ be the locally $\C^{\x}$-finite part of $\CW'_\hbar$ . Then the finite $W$-algebra $\CW$ is defined by $\CW:=\CW_\hbar/(\hbar-1)\CW_\hbar$. There is an embedding $Z \hookrightarrow \CW$ which becomes an isomorphism when $e$ is regular.

Let $M$ be an object in $\HC(U)$. Let $\{F_i M\}_{i \geq 0}$ be a good filtration on $M$ and consider the Rees module $M_\hbar $ for this good filtration. Let $M^{\wedge_\xi}_\hbar $ be the completion at the kernel of $\xi: \CU_\hbar \rightarrow \C$. Then we will have an isomorphism of $\CU_\hbar^{\wedge_\xi}\cong \CA_\hbar^{\wedge_0}\widehat{\otimes}_{\C[[\hbar]]} \CW'_\hbar$-bimodules \cite[Proposition $3.3.1$]{IL11}
\[ M^{\wedge_\xi}_\hbar \cong \CA_\hbar^{\wedge_0}\widehat{\otimes}_{\C[[\hbar]]} \CM'_\hbar.\]
The $\CW'_\hbar$-bimodule $\CM'_\hbar$ comes with a $\C^{\x}$-action and a $C(G)$-action. Note that $C(G)$ acts trivially on $\CA^{\wedge_0}_\hbar$.  Let $\CM_\hbar$ be the locally $\C^{\x}$-finite part of $\CM'_\hbar$.  Then the image of $M$ under the functor $\bullet_\dag$ is defined by:
\[ M_\dag :=\CM_\hbar/(\hbar-1)\CM_\hbar,\]
which is an object in $Z\bim^{C(G)}$. 

The functor $\bullet_\dag$ satisfies the following properties \cite{IL11}:

\begin{Prop}\label{prop property of restriction}
a) $\bullet_\dag$ is exact and  monoidal.

b)  $\bullet^\dag$ is right adjoint to $\bullet_\dag: \HC_\CN(U)\rightarrow Z\textnormal{-bimf}^{C(G)}$.

c) The kernel and cokernel of the natural map $\epsilon: M\rightarrow (M_\dag)^\dag$ lie in $\HC_{\partial\CN(U)}$.

d)  Both functors $\bullet_\dag$ and $\bullet^\dag$ are $Z\otimes Z$-linear; i.e., for $M_1, M_2$ in $\HC(U)$ and $N_1, N_2$ in $Z\textnormal{-bimf}^{C(G)}$, the following maps are $Z\otimes Z$-linear:
\Eqn{ \Hom_{\HC(U)}(M_1, M_2)\rightarrow \Hom_{Z\bim^{C(G)}}(M_{1,\dag}, M_{2, \dag})\\
\Hom_{Z\textnormal{-bimf}^{C(G)}}(N_1,N_2)\rightarrow \Hom_{\HC(U)}(N_1^\dag, N_2^\dag)}

e) If $C(G)$ acts on $M$ via a character $\chi$, then $C(G)$ also acts on $M_\dag$ via the character $\chi$.

\end{Prop}

Let $I$ be an ideal in $Z\otimes Z$. Then for any $M$ in $\HC(U)$, both $IM$ and $M/I$ are objects contained in $\HC(U)$.
\begin{Cor}\label{cor: quotient and restriction functor} We have isomorphisms: $(MI)_\dag\cong M_\dag I$ and $(M/I)\dag\cong M_\dag/I$.
\end{Cor}
% \begin{proof} Since $Z \otimes Z$ is Noetherian, $I$ is finitely generated over $\CZ$ so that we can choose generators $x_1, \dots, x_n$ for $I$. For each $x_i$ , we form the morphism $m_{x_i}: M\xrightarrow[]{\cdot x_i} M$. It can be seen that $MI$ is just the image of the natural map
% \[ \bigoplus_1^n M\xrightarrow[]{\oplus_1^n m_{x_i}} M.\]
% The exactness and $Z\otimes Z$-linearity of $\bullet_\dag$ imply that $(MI)_\dag \cong M_\dag I$. The exact sequence $0\rightarrow IM \rightarrow M \rightarrow M/I \rightarrow 0$ and the exactness of $\bullet_\dag$ imply that $(M/I)_\dag \cong M_\dag/I$.
% \end{proof}

We are now interested in the adjunction unit $\epsilon: M\rightarrow (M_\dag)^\dag$ for $M$ in $\HC_\CN(U)$. Let $\iota: \BO_{\reg}\hookrightarrow \CN$ be the inclusion of the regular nilpotent orbit $\BO_{\reg}$ into the nilpotent cone $\CN$. Let $F$ be a good filtration on $M$ such that $\gr_F M$ is a module over $\C[\CN]$. Then there is a suitable filtration on $(M_\dag)^\dag$ such that the associated graded map $\gr \epsilon: \gr  M \rightarrow \gr  (M_\dag)^\dag$ can be viewed as the natural map $\gr  M\rightarrow \iota_*\iota^* \gr M$, see Section $3.2$ and the construction of $\bullet_\dag$ in \cite{IL11}. On the other hand, $\CN$ is normal and the codimension of $\CN\backslash\BO_{\reg}$ in $\CN$ is $2$. Therefore, if $\gr M$ is a free $\C[\CN]$-module then the natural map $\gr M\rightarrow \iota_*\iota^*\gr M$ is an isomorphism and the map $\epsilon: M\rightarrow (M_\dag)^\dag$ is also an isomorphism. The last observation is crucial for the following lemma. 

\begin{Lem}\label{lem ff of restriction} Let $V$ be a finite dimensional $\g$-representation. Let $\m_\chi$ be a maximal ideal of $Z$. The natural map $V\otimes U(\g)/\m_\chi^k \rightarrow ((V\otimes U(\g)/\m_\chi^k)_\dag)^\dag$ is an isomorphism for any $k$.
\end{Lem}
\begin{proof}To simplify the notation in the proof, we set $\mathcal{V}_k:=V\otimes U(\g)/\m_\chi^k$. We first prove the lemma for the case $k=1$. The PBW filtration on $U(\g)$ gives a filtration on $\mathcal{V}_1$ so that $\gr  \mathcal{V}_1$ is isomorphic to $\C[\CN]\otimes V$ as $\C[\CN]$-modules. Therefore, by the observation in the preceding paragraph, we see that the natural map $\mathcal{V}_1\rightarrow (\mathcal{V}_{1,\dag})^\dag$ is an isomorphism.

We now prove the lemma by induction. Let assume that $\mathcal{V}\rightarrow (\mathcal{V}_{k, \dag})^\dag$ is an isomorphism. Since $U(\g)$ is free over $Z$, the Harish-Chandra bimodule $V\otimes (U(\g)\m_\chi^k/U(\g)\m_\chi^{k+1})$ is isomorphic to the direct sum of finitely many copies of $V\otimes U(\g)/\m_\chi$. Therefore the map $V\otimes U(\g)\m_\chi^k/\m_\chi^{k+1}\rightarrow ((V\otimes U(\g)\m_\chi^k/\m_\chi^{k+1})_\dag)^\dag$ is an isomorphism. Let us consider the following diagram
\[\begin{tikzcd} 0\arrow[r]& V\otimes U(\g)\m_\chi^k/\m_\chi^{k+1}\arrow[r]\arrow[d]&\mathcal{V}_{k+1}\arrow[r]\arrow[d]&\mathcal{V}_k \arrow[r]\arrow[d]&0\\
0\arrow[r]&((V\otimes U(\g)\m_\chi^k/\m_\chi^{k+1})_\dag)^\dag \arrow[r]& (\mathcal{V}_{k+1, \dag})^\dag \arrow[r]& (\mathcal{V}_{k, \dag})^\dag& \end{tikzcd}\]

The left and right vertical maps are isomorphisms. Moreover, the functor $(\bullet_\dag)^\dag$ is left exact. Therefore, the middle vertical map is also an isomorphism.
\end{proof}

\begin{Cor}\label{cor: ff of restriction} Let $P$ be a direct summand of the object of the form $V\otimes U(\g)/\m_\chi^k$ as in \textnormal{Lemma} \ref{lem ff of restriction}. Then for any $M$ in $\HC_\CN(U)$, the following map is bijective
\[\Hom(M,P)\rightarrow \Hom(M_\dag, P_\dag).\]
\end{Cor}

%%%%%%%%%%%%%%%%%
\subsection{The functor $\bullet_\dag: \HC(U^{\chi, \chi'})\rightarrow (Z^{\wedge_\chi}, Z^{\wedge_{\chi'}})\bim^{C(G)}$.}\label{ssec: complete-restriction-functor}\

Let $(Z^{\wedge_\chi}, Z^{\wedge_{\chi'}})\bim^{C(G)}$ denote the category of $\mathfrak{X}$-graded $(Z^{\wedge_\chi}, Z^{\wedge_{\chi'}})$-bimodules which are finitely generated as left $Z^{\wedge_\chi}$-modules and finitely generated as right $Z^{\wedge_{\chi'}}$-modules; here $\mathfrak{X}$ is the set of characters of $C(G)$. We now construct the following functor based on \eqref{Restriction functor}:
\Eq{\label{complete restriction}\bullet_\dag: \HC(U^{\chi, \chi'})\rightarrow (Z^{\wedge_\chi}, Z^{\wedge_{\chi'}})\bim^{C(G)}.}

Let $M$ be an object in $\HC(U^{\chi, \chi'})$. Then $M/M\m_{\chi'}^k$ is an object  in $\HC_\CN(U)$ so that $(M/M\m_{\chi'}^k)_\dag$ is contained in $Z\bim^{C(G)}$. We define:
\[ M_\dag:=\varprojlim (M/M \m_{\chi'}^k)_\dag.\]
For any $k$, the object $M/M \m_{\chi'}^k$ is killed by $\m_{\chi}^l$ on the left for some $l>0$. Therefore, $M_\dag$ carries a natural $(Z^{\wedge_\chi}, Z^{\wedge_{\chi'}})$-bimodule structure. The $C(G)$-action on $M$ natually induces an $C(G)$-action on $M_\dag$, which gives a $\mathfrak{X}$-grading on $M_\dag$. So $M_\dag$ is an object in $(Z^{\wedge_\chi}, Z^{\wedge_{\chi'}})\bim^{C(G)}$. Furtheremore, we will get the naturally isomorphic bimodule if we define:
\[ M_\dag:=\varprojlim (M/\m_{\chi}^kM)_\dag.\]
At the moment, we only have that  $M_\dag$ is a $\mathfrak{X}$-graded $(Z^{\wedge_\chi}, Z^{\wedge_{\chi'}})$-bimodule, but we will prove later in Proposition \ref{prop property of complete restriction} that $M_\dag$ is finitely generated as a left $Z^{\wedge_\chi}$-module and finitely generated as a right $Z^{\wedge_{\chi'}}$-module. 

 Now we go through  some properties of the functor $\bullet_\dag$ in \eqref{complete restriction}.

\begin{Prop}\label{prop property of complete restriction} 
a) The functor $\bullet_\dag$   is exact. 

b) The functor $\bullet_\dag$  is $Z^{\wedge_{\chi'}}$-linear and $Z^{\wedge_\chi}$-linear.

c) For any ideal $I$ in $Z^{\wedge_\chi}\otimes Z^{\wedge_{\chi'}}$, we have $(M/I)_\dag\cong  M_\dag/I$ and $ (MI)_\dag\cong M_\dag I$.

d) Recall that for any $M$ in $\HC(U)$ the object $M\otimes_\CZ \CZ^{\chi, \chi'}$ is in $\HC(U^{\chi, \chi'})$. Then $(M\otimes_\CZ \CZ^{\chi, \chi'})_\dag \cong M_\dag\otimes_\CZ\CZ^{\chi, \chi'}$.

e) For any $M$ in $\HC(U^{\chi, \chi'})$, the image $M_\dag$ is finitely generated as a left $Z^{\wedge_\chi}$-module and as a right $Z^{\wedge_{\chi'}}$-module.

f) Let $M$ be in $\HC(U^{\chi, \chi'})$. If $C(G)$ acts on $M$ via a character $\chi$ then $C(G)$ also acts on $M_\dag$ via the character $\chi$.
\end{Prop}

To establish the exactness of the functor $\bullet_\dag$ in \eqref{complete restriction} , we will use the following technical lemma:
\begin{Lem} \label{prop gemerator complete V}There is a finite dimensional $\g$-representation $\hat{V}$ such that for any $K\in \HC(U^{\chi, \chi'})$, there is a functorial inclusion $K_\dag \hookrightarrow \Hom_\g(\hat{V},K)$.
\end{Lem}

\begin{proof} Let $_\chi\HC^k_{\chi'}$ be the full subcategory of $\HC(U)$ consisting of all Harish-Chandra bimodules $M$ such that $M\m_{\chi'}^k=0,  \m_{\chi}^l M=0$ for some $l>0$. By the Bernstein-Gelfand equivalence \cite{BG80}, the category $_\chi\HC^k_{\chi'}$ is a $\C$-finite abelian category, i.e., an abelian category whose Hom sets between objects are all finite dimensional $\C$-vector spaces. The functor $\bullet_\dag: \;_\chi\HC^k_{\chi'}\rightarrow Z\bim^{C(G)}\xrightarrow[]{\textnormal{forget}} \textnormal{Vecf}$ is exact; here $\textnormal{Vecf}$ is the category of finite dimensional $\C$-vector spaces. Therefore, there is a  unique up to isomorphism projective object $P_k$ in $_\chi \HC^k_{\chi'}$ such that the functor $\bullet_\dag: \;_\chi\HC^k_{\chi'}\rightarrow \text{Vecf}$ can be represented by $\Hom(P_k, \bullet)$. One can easily show that $P_{k+1}/P_{k+1}\m_{\chi'}^k $ is projective and also represents the exact functor $\bullet_\dag: \; _\chi \HC^{k+1}_{\chi'} \rightarrow Z\bim^{C(G)} \xrightarrow[]{\text{forget}} \text{Vecf}_\C$, hence $P_{k+1}/P_{k+1}\m_{\chi'}^k \cong P_k$.

Let $\hat{V}$ be a finite dimensional $\g$-representation that generates $P_1$  as a $U(\g)$-bimodule. We have a surjective map $_\chi(\hat{V}\otimes U(\g)/\m_{\chi'})\rightarrow P_1$. We can find a system of finite dimensional $\g$-representations $V_k \subset P_k$ such that $V_1=\hat{V}$ in $P_1$ and  $V_{k+1}$ maps isomorphically onto $V_k$ under the surjective map $P_{k+1}\rightarrow P_k$. Using Lemma $\ref{cor: surj maps}$ and the fact that $_\chi(V_1\otimes U(\g)/\m_{\chi'})\rightarrow P_1$ is surjective  by the construction, we have that the following maps
\[_\chi(V_k\otimes U(\g)/\m_{\chi'}^k) \rightarrow P_k\]
are surjective for any $k$. 

Now the lemma follows:
\[ K_\dag=\varprojlim (K/K\m_{\chi'}^k)_\dag= \varprojlim \Hom(P_k, K/K\m_{\chi'}^k) \hookrightarrow \varprojlim \Hom(\hat{V}, K/K\m_{\chi'}^k)=\Hom_\g(\hat{V}, K).\]
The functoriality of this inclusion follows by the construction.
\end{proof}

\begin{proof}[Proof of Proposition \ref{prop property of complete restriction}]

a)  Let $0\rightarrow L\rightarrow M\rightarrow N\rightarrow 0$ be a short exact sequence in $\HC(U^{\chi, \chi'})$. It gives us an exact sequence
\Eq{\label{logn exact sequence} 0\rightarrow K_k \rightarrow L/L\m_{\chi'}^k\rightarrow M/M\m_{\chi'}^k \rightarrow N/N\m_{\chi'}^k \rightarrow 0.}
for each $k$. The last exact sequence can be viewed as an exact sequence in $\HC(U)$, therefore, we get an exact sequences of inverse  systems
\[ 0\rightarrow \{ K_{k, \dag}\}\rightarrow \{(L/L\m_{\chi'}^k)_\dag\}\rightarrow \{M/M\m_{\chi'}^k)_\dag\}\rightarrow \{(N/N\m_{\chi'}^k)_\dag\} \rightarrow 0.\]
Note that these inverse limit systems consist of  finite dimensional vector spaces. Therefore, we have an exact sequence:
\[0\rightarrow \varprojlim K_{k, \dag} \rightarrow L_\dag\rightarrow M_\dag \rightarrow N_\dag\rightarrow 0.\]
So it is enough to show that $\varprojlim K_{k, \dag} =0$. Let us show that $\varprojlim \Hom_\g(V, K_k)=0$ for any finite dimensional $\g$-representation $V$. By \eqref{logn exact sequence}, we also have an exact sequences of inverse systems:
\[ 0\rightarrow \Hom_\g(V, K_k)\rightarrow \Hom_\g(V, L/L\m_{\chi'}^k) \rightarrow \Hom_\g(V, M/M\m_{\chi'}^k \rightarrow \Hom_\g(V, N/N\m_{\chi'}^k)\rightarrow 0.\]
Since $V$ is finite dimensional, these inverse limit systems gives us an exact sequence:
\Eq{\label{short exact of Hom} 0\rightarrow \varprojlim \Hom_\g(V, K_k)\rightarrow \varprojlim \Hom_\g(V, L/L\m_{\chi'}^k)\rightarrow \varprojlim \Hom_\g(V, M/M\m_{\chi'}^k)\\ \notag \rightarrow \varprojlim \Hom_\g(V, N/N \m_{\chi'}^k) \rightarrow 0.}
By Corrolary \ref{prop of HCchi}, we have $\varprojlim \Hom_\g(V, L/L\m_{\chi'}^k)=\Hom_\g(V,L)$, similarly with last two terms in \eqref{short exact of Hom}. Therefore, \eqref{short exact of Hom} becomes
\[ 0\rightarrow  \varprojlim \Hom_\g(V, K_k)\rightarrow \Hom_\g(V, L) \rightarrow \Hom_\g(V, M) \rightarrow \Hom_\g(V, N)\rightarrow 0,\]
but $0\rightarrow \Hom_\g(V, L) \rightarrow \Hom_\g(V, M)\rightarrow \Hom_\g(V,N)\rightarrow 0$ is an exact sequence. Therefore,  $\varprojlim \Hom_\g(V, K_k)=0$ for any finite dimensional $\g$-representation $V$.

Now Lemma \ref{prop gemerator complete V} gives us an injective map $ \varprojlim K_{k, \dag}\hookrightarrow \varprojlim \Hom_\g(\hat{V}, K_k)$. Therefore, $\varprojlim K_{k, \dag}=0$.

b)  This is because the functor $\bullet_\dag: \HC(U) \rightarrow Z\bim^{C(G)}$ is $Z\otimes Z$-linear.

c) The proof is similar to that of the Lemma $\ref{cor: quotient and restriction functor}$.

d) Let $I_{\chi, \chi'}$ be the maximal ideal of the closed point $(\chi, \chi')$ in Spec $(Z\otimes Z)$. Then 
\[ (M\otimes_\CZ\CZ^{\chi, \chi'})_\dag\cong \varprojlim (M/MI^k_{\chi, \chi'})_\dag \cong \varprojlim M_\dag/I^k_{\chi, \chi'} \cong M_\dag \otimes_\CZ\CZ^{\chi, \chi'}.\]
The last isomorphism holds because $M_\dag$ is finitely generated over $\CZ$.

e) There is a surjective map $P^{\chi, \chi'}_V \twoheadrightarrow M$. The exactness of $\bullet_\dag$ gives us a surjective map $(P^{\chi, \chi'}_V)_\dag \twoheadrightarrow M_\dag$. Then the statement follows from part $d)$.

f) Follows by  Proposition \ref{prop property of restriction}.

\end{proof}

\begin{Prop}\label{prop monoidal}
For any $M$ in $\HC(U^{\chi'', \chi})$ and $N$ in $\HC(U^{\chi, \chi'})$, there is a natural isomorphism
\Eq{\label{monoidal iso} (M\otimes_{U^{\chi}}N)_\dag \cong M_\dag\otimes_{Z^{\wedge_\chi}} N_\dag.}
\end{Prop}
\begin{proof}
We will first  prove that:
\Eq{\label{monoidal iso 1}M_\dag\otimes_{Z^{\wedge_\chi}}N_\dag &\cong \varprojlim M_\dag/M_\dag\m_{\chi}^k \otimes_Z N_\dag/\m_{\chi}^kN_\dag,\\
\label{monoidal iso 2}(M\otimes_{U^\chi}N)_\dag&\cong \varprojlim (M/M\m_{\chi}^k\otimes_UN/\m_{\chi}^k N)_\dag.}

To prove \eqref{monoidal iso 1} we need the following claim: Let $(R,\m)$ be a commutative complete local Noetherian ring and $M,N$ be finitely generated $R$-modules. Then $M\otimes_R N \cong \varprojlim M/\m^k \otimes_R N/\m^k$. Indeed, the claim holds if $N$ is free. In general, let $F_2\rightarrow F_1\rightarrow N\rightarrow 0$ be a free presentation of $N$. The following commutative diagram implies that the claim holds whence $M$ and $N$ are finitely generated $R$-modules:
\[ \begin{tikzcd} M\otimes_R F_2\arrow[d]\arrow[r]& M\otimes_R F_1\arrow[r]\arrow[d]& M\otimes_R N \arrow[r]\arrow[d]&0\\
\varprojlim M/\m^k\otimes_R F_2/\m^k\arrow[r]&\varprojlim M/\m^k \otimes_R F_1/\m^k \arrow[r]& \varprojlim M/\m^k \otimes_R N/\m^k \arrow[r]&0\end{tikzcd}\]
The claim implies \eqref{monoidal iso 1} since $M_\dag$ is a finitely generated right $Z^{\wedge_\chi}$-module and $N_\dag$ is a finitely generated left $Z^{\wedge_\chi}$-module. Note that $M_\dag/M_\dag \m_{\chi}^k \otimes_{Z^{\wedge_\chi}} N_\dag/\m_{\chi}^k N_\dag \cong M_\dag/M_\dag\m_{\chi}^k \otimes_Z N_\dag/\m_{\chi}^k N_\dag$.

It is enough to prove \eqref{monoidal iso 2}  in the case $N=P^{\chi, \chi'}_V$, since $\bullet_\dag$ is exact and any object in $\HC(U^{\chi, \chi'})$ can be covered by some projective object of the form $P^{\chi, \chi'}_V$. By construction, 
\[ (M\otimes_{U^\chi} P^{\chi, \chi'}_V)_\dag :=\varprojlim ((M\otimes_{U^\chi}P^{\chi, \chi'}_V)/\m_{\chi''}^k (M\otimes_{U^\chi} P^{\chi, \chi'}_V))_\dag.\]
Since $P^{\mu,\lambda}_V$ is projective as a left $U^\chi$-module, we have:
\[ (M\otimes_{U^\chi}P^{\chi, \chi'}_V)_\dag\cong \varprojlim (M/\m^k_{\chi''}M\otimes_{U^\chi} P^{\chi, \chi'}_V)_\dag.\]
On the other hand, the $\m_{\chi''}$-adic topology and the $\m_{\chi}$-adic topology on $M$ are equivalent, therefore:
\[ (M\otimes_{U^\chi}P^{\chi, \chi'}_V)_\dag\cong \varprojlim (M/M\m_{\chi}^k \otimes_{U^\chi} P^{\chi, \chi'}_V)_\dag \cong \varprojlim (M/M\m_{\chi}^k \otimes_U P^{\chi, \chi'}_V/\m_{\chi}^k P^{\chi, \chi'}_V)_\dag.\]
So we proved \eqref{monoidal iso 2}.

Now since $\bullet_\dag: \HC(U)\rightarrow Z\bim^{C(G)}$ is monoidal and by Lemma $\ref{cor: quotient and restriction functor}$, we have 
\[(M/M\m_{\chi}^k \otimes_U N/\m_{\chi}^k N)_\dag\cong (M/M\m_{\chi}^k)_\dag \otimes_Z (N/\m_{\chi}^k N)_\dag \cong M_\dag/M_\dag\m_{\chi}^k \otimes_Z N_\dag/\m_{\chi}^k N_\dag.\]
Combining with \eqref{monoidal iso 1} and \eqref{monoidal iso 2}, we get the natural isomorphism \eqref{monoidal iso}.
\end{proof}
\begin{Rem} The naturality of the isomorphisms in Proposition \ref{prop monoidal} says that $\bullet_\dag$ is monoidal with the  monoidal structure in Section $\ref{ssec:  complete HCbim}$. 
\end{Rem}

\begin{Prop}\label{prop ff of complete proj} Let $P$ be a projective object of $\HC(U^{\chi, \chi'})$ and $M$ be any object  of $\HC(U^{\chi, \chi'})$. Then the following map is bijective:
\Eq{\label{fully faithful iso}\Hom_{\HC(U^{\chi, \chi'})}(M,P)\rightarrow \Hom_{(Z^{\wedge_\chi}, Z^{\wedge_{\chi'}})\bim^Z}(M_\dag, P_\dag).}
\end{Prop}
\begin{proof} It is enough to prove the  statement for a projective object of the form $P^{\chi, \chi'}_V$ as in \eqref{Def of proj bimod}.  We have $P^{\chi, \chi'}_V /P^{\chi, \chi'}_V \m_{\chi'}^k= \;_\chi(V\otimes U(\g)/\m_{\chi'}^k)$, the direct summand of $V\otimes U(\g)/\m_{\chi'}^k$ whose support in Spec $Z\otimes Z$ is $(\chi, \chi')$. By Corrollary $\ref{cor: ff of restriction}$, for any $k$, the following map is bijective:
\[ \Hom_{\HC(U)}(M/M\m_{\chi'}^k, P^{\chi, \chi'}_V/P^{\chi, \chi'}_V \m_{\chi'}^k)\rightarrow \Hom_{Z\bim^Z}((M/M\m_{\chi'}^k)_\dag, (P^{\chi, \chi'}_V/P^{\chi, \chi'}_V \m_{\chi'}^k)_\dag).\]
These bijective maps form a map between two inverse systems so that the map \eqref{fully faithful iso} is the inverse limit of these maps. Therefore, the map \eqref{fully faithful iso} is bijective.
\end{proof}

Let  $(\mu, \lambda)$ be a pair of two dominant weights such that $\mu\in \chi, \lambda \in \chi'$ and $\lambda-\mu \in \Lambda$. Assume that $W_\lambda \subset W_\mu$. We recall the inclusion of algebras  $Z^{\wedge_\chi}\xrightarrow[]{\epsilon_\mu}R^{W_\mu}\hookrightarrow R^{W_\lambda}\xrightarrow[]{\epsilon^{-1}_\lambda} Z^{\wedge_{\chi'}}$. By this inclusion, we can view $Z^{\wedge_{\chi'}}$ as a $(Z^{\wedge_\chi}, Z^{\wedge_{\chi'}})$-bimodule and as a $(Z^{\wedge_{\chi'}},Z^{\wedge_\chi})$-bimodule.  Recall the translation bimodules $P^{\mu, \lambda}\in \HC(U^{\chi, \chi'})$ and $P^{\lambda, \mu}\in \HC(U^{\chi', \chi})$.

\begin{Prop} \label{prop image of complete tranbimod}Let  $(\mu, \lambda)$ be a pair of two dominant weights such that $\mu\in \chi, \lambda \in \chi'$,  $\lambda-\mu \in \Lambda$, and $W_\lambda \subset W_\mu$. We have an isomorphism $P^{\mu, \lambda}_\dag \cong Z^{\wedge_{\chi'}}$ as $(Z^{\wedge_\chi}, Z^{\wedge_{\chi'}})$-bimodules. Furthermore, $C(G)$ acts on $P^{\mu, \lambda}_\dag$ by the same character as on $P^{\mu, \lambda}$. There are similar statements for $P^{\lambda, \mu}_\dag$.
\end{Prop}

To prove Proposition \ref{prop image of complete tranbimod}, we need the following lemma whose proof will be provided in Section $\ref{sec7}$. We set $P^{\mu, \ulambda}:=\:_\chi(L(\mu-\lambda)\otimes U(\g)/\m_{\chi'})$, the direct summand of the Harish-Chandra bimodule $L(\mu-\lambda)\otimes U(\g)/\m_{\chi'}$ whose support in Spec $Z \otimes Z$ is $(\chi, \chi')$. Then $P^{\mu, \ulambda}\cong P^{\mu, \lambda}/P^{\mu, \lambda}\m_{\chi'}$. Similarly, we set $P^{\ulambda, \mu}=(U(\g)/\m_{\chi'} \otimes L(\lambda-\mu))_\chi$ and then $P^{\ulambda, \mu}\cong  P^{\lambda, \mu}/\m_{\chi'}P^{\lambda, \mu}$.

\begin{Lem}\label{lem image of complete tranbimod}Let  $(\mu, \lambda)$ be a pair of two dominant weights such that $\mu\in \chi, \lambda \in \chi', \lambda-\mu \in \Lambda$ and $W_\lambda \subset W_\mu$. Then the image $P^{\mu, \ulambda}_\dag$ is isomorphic to $Z^{\wedge_{\chi'}}/\m_{\chi'}\cong \C$ as right $Z^{\wedge_{\chi'}}$-modules. Similarly, $P^{\ulambda, \mu}_\dag \cong \C$ as left $Z^{\wedge_{\chi'}}$-modules.
\end{Lem}

\begin{proof}[Proof of Proposition \ref{prop image of complete tranbimod}] By Proposition \ref{prop actions on tranbimod}, it is enough to show that $P^{\mu, \lambda}_\dag \cong Z^{\wedge_{\chi'}}$ as right $Z^{\wedge_{\chi'}}$-modules. By the definition of $P^{\mu, \lambda}$ , the right action of $Z^{\wedge_{\chi'}}$ on $P^{\mu, \lambda}$ is torsion free. On the other hand, the functor $\bullet_\dag$ is exact and $Z^{\wedge_\chi}\otimes Z^{\wedge_{\chi'}}$-bilinear. Therefore, $P^{\mu, \lambda}_\dag$ is also torsion free as a right $Z^{\wedge_{\chi'}}$-module. 

So we have that $P^{\mu, \lambda}_\dag$ is finitely generated and torsion free as a right   $Z^{\wedge_{\chi'}}$-module and $P^{\mu, \lambda}_\dag/P^{\mu, \lambda}_\dag \m_{\chi'}$ is isomorphic to $\C$ by Lemma \ref{lem image of complete tranbimod}. These imply that $P^{\mu, \lambda}\cong Z^{\wedge_{\chi'}}$ as right $Z^{\wedge_{\chi'}}$-modules, completing the proof.
\end{proof}

% \begin{Rem} \label{rem4}We recall the group $\widehat{W}_{[\lambda]}$ and the semi-direct product $\widehat{W}_{[\lambda]}= C\ltimes W_{[\lambda]}$ in Section $\ref{non integral weight}$. Let $w\in C$ then  $w\cdot \mu$ is a dominant weight such that $w\cdot \mu-\mu \in \Lambda$. Assume that $W_\lambda \subset W_{w\cdot \mu}$. We will need the $(R^{W_\mu}, R^{W_{w\cdot\mu}})$-bimodule $R^{I_w}$ defined in Section \ref{ssec:  another reaization of B(mu, lambda)}. By the inclusion $R^{W_{w\cdot\mu}}\hookrightarrow R^{W_\lambda}$,  we have a natural $(R^{W_{w\cdot \mu}}, R^{W_\lambda})$-bimodule structure on $R^{W_\lambda}$. If we use the identifications $Z^{\wedge_\chi}\overset{\epsilon_{w\cdot\mu}}{\cong} R^{W_{w\cdot \mu}}$ and $Z^{\wedge_{\chi'}}\overset{\epsilon_\lambda}{\cong}R^{W_{\lambda}}$ then we have an isomorphism of $(R^{W_{w\cdot\mu}}, R^{W_\lambda})$-bimodules
% \[ P^{w\cdot\mu, \lambda}_\dag \cong R^{W_\lambda}.\]
% On the other hand, if we use the identifications $Z^{\wedge_\chi}\overset{\epsilon_\mu}{\cong}R^{W_\mu}$ and $Z^{\wedge_{\chi'}}\overset{\epsilon_\lambda}{\cong}R^{W_\lambda}$ then we have an isomorphism of $(R^{W_\mu}, R^{W_\lambda})$-bimodules
% \[ P^{w\cdot\mu, \lambda}_\dag\cong R^{I_w}\otimes_{R^{W_{w\cdot \mu}}} R^{W_\lambda}.\]
% \end{Rem}

\subsection{Proof of Lemma \ref{lem image of complete tranbimod}} \label{sec7}\

Let $\lambda$ be a weight in $\fh^*$. Let $O_{\lambda+\Lambda}$ be the full subcategory of $O$ consisting of all objects whose weights lie in $\lambda+\Lambda$. Recall $C(G)$ is the center of the group $G$ and $Z$ is the center of $U(\g)$. In \cite[$\mathsection 4.1$]{IL15}, there is the functor
\Eq{\label{restriction for O}\bullet_\dag: O_{\lambda+\Lambda}\rightarrow Z\text{-mod}^{C(G)};}
here $Z\text{-mod}^{C(G)}$ is the category of $\mathfrak{X}$-graded $Z$-modules. 

Let us briefly describe how this functor is constructed in ~\cite{IL15}. Let $M$ be an object in $O_{\lambda+\Lambda}$. For each $\lambda \in \fh^*$, there is  the corresponding character $\lambda: \fb\rightarrow \C$. We define the $\lambda$-shifted action of $\fb$ on $M$ as follows: $x\cdot_\lambda m=xm -\lambda(x)m$ for any $x\in \fb$ and $m \in M$. This $\lambda$-shifted action can be integrated into an $B$-action on $M$. The action of $C(G)\subset B$ on each object in $O_{\lambda+\Lambda}$ will change if we replace the character $\lambda$ by another character $\lambda'\in \lambda+\Lambda$ in the shift action of $\fb$ . Therefore, we will fix the  shift action of $\fb$ by a fixed character $\lambda$.

We can find some \emph{good filtration} $\{F_iM\}$ on $M$ such that $[F_iU, F_jM]\subset F_{i+j-1}M$ and this filtration is $B$-stable. The Rees module $R_\hbar M$ is a module over $\CU_\hbar$, so we can consider the completion $R_\hbar^{\wedge_\xi} M$ at the character $\xi: \CU_\hbar \rightarrow \C$.

Recall the Kirillov-Kostant skew-symmetric form on $T_\xi\g^*\cong \g$. The subspace $\fb \subset \g$ is a maximal isotropic subspace of $T_\xi\g^*$, and $\mathfrak{u}:=\fb\cap T_\xi \BO$ is a Langrangian subspace of the symplectic vector space $T_\xi \BO$. The space $\C[[\hbar, \mathfrak{u}]]$ is naturally a module over the completed Weyl algera $\CA^{\wedge_0}_\hbar$. A lifting of $V:=T_\xi \BO$ to $\CU_\hbar$ can be chosen so that the image of $\mathfrak{u}$ is contained in $\fb\CU_\hbar$, \cite[Lemma $4.1$]{IL15}. Then we have a decomposition of $\CU_\hbar^{\wedge_\xi}\cong \CA_\hbar^{\wedge_0}\widehat{\otimes}_{\C[[\hbar]]} \CW'_\hbar$-modules:
\[ R_\hbar^{\wedge_\xi}M \cong \C[[\hbar, \mathfrak{u}]]\widehat{\otimes} \CM'_\hbar.\]
The space $\CM'_\hbar$ comes with an $\C^{\x}$-action. Let $\CM_\hbar$ denote the locally $\C^{\x}$-finite part of $\CM'_\hbar$. We set:
\[ M_\dag:= \CM_\hbar/(\hbar-1)\CM_\hbar.\]

Let $\fb^\perp$ be the annihilator of $\fb$ in $\g^*$. The associated graded $\gr _F M$ is a module over $S(\g)/S(\g)\fb\cong \C[\fb^\perp]$. Define $\text{mult}_{\fb^\perp} M$ to be the multiplicity of $\gr _F M$ on the irreducible variety $\fb^\perp$.

The functor in \eqref{restriction for O} satisfies the following properties, see \cite[$\mathsection4.1.4$, $\mathsection 4.2.4$]{IL15}:
\begin{Prop}
a) The functor $\bullet_\dag$ in \eqref{restriction for O} is exact.

b) If $M\in \HC(U)$ and $ X\in O_{\lambda+\Lambda}$, then $(M\otimes_U X)_\dag \cong M_\dag\otimes_Z X_\dag$.

c) For any $X$ in $O_{\lambda+\Lambda}$, the image $X_\dag$ is a finite dimensional vector space over $\C$. Furtheremore, $\dim_\C X_\dag =\textnormal{mult}_{\fb^\perp}X$.

d) For any $X, Y$ in $O_{\lambda+\Lambda}$, the following map:
        \[ \Hom_O(X,Y)\rightarrow \Hom_{Z\textnormal{-mod}^{C(G)}}(X_\dag, Y_\dag)\]
        is $Z$-linear.
\end{Prop}
\begin{Rem}
    There is also an analog of the functor $\bullet_\dag$ in \eqref{restriction for O} for the category $O^r_{\lambda+\Lambda}$ of the right $U(\g)$-modules; see the end of Section $4.1.3$ in \cite{IL15}.
\end{Rem}
\begin{Cor}
    For any $\mu \in \lambda+\Lambda$, we have $\Delta(\mu)_\dag\cong \C$ as $Z/\m_{\chi_\mu}$-modules.
\end{Cor}
\begin{proof}
We have $\dim_\C\Delta(\mu)_\dag=\text{mult}_{\fb^\perp} \Delta(\mu)=1$. On the other hand, $\Delta(\mu)_\dag$ is killed by some power of $\m_{\chi_\mu}$. This implies the corollary.
\end{proof}

\begin{proof}[Proof of Lemma \ref{lem image of complete tranbimod}] Since $P^{\mu, \ulambda}_\dag$ is a right $Z/\m_{\chi'}$-module, it is enough to show that $\dim_\C P^{\mu, \ulambda}_\dag=1$. We have $P^{\mu,\ulambda}\otimes_U \Delta(\lambda)\cong \Delta(\mu)$.  Therefore, $P^{\mu, \ulambda}_\dag\otimes_Z \Delta(\lambda)_\dag\cong \Delta(\mu)_\dag$. On the other hand, since $P^{\mu, \ulambda}_\dag$ is a right $Z/\m_{\chi'}$-module and $\Delta(\lambda)_\dag\cong Z/\m_{\chi'}$ as left $Z$-modules, we have that  $P^{\mu, \ulambda}_\dag \otimes_Z \Delta(\lambda)_\dag \cong P^{\mu, \ulambda}_\dag$ as left $Z$-modules. Therefore, $P^{\mu, \ulambda}_\dag\cong \Delta(\mu)_\dag \cong \C$ as vector spaces.

The proof for $P^{\ulambda, \mu}_\dag\cong Z/\m_{\chi'}$ as left $Z$-modules is similar, but we need to consider the category $O^r_{\lambda+\Lambda}$ instead of the category $O_{\lambda+\Lambda}$.
\end{proof}

\subsection{Another realization of the category $\CB_{\mu, \lambda}$} \label{ssec:  another reaization of B(mu, lambda)}\ 

We will give another realization to the category $\CB_{\mu, \lambda}$ defined in the introduction.  For any subset $I$ of the set of simple reflections $\Pi_{[\lambda]}$ in $W_{[\lambda]}$, let $W_I$ be the subgroup generated  by the elements in $I$. Let $R^I$ be the $W_I$-invariant part of $R$ under the usual action of $W_{[\lambda]}$. We have
\[R_w\otimes_R R\otimes_{R^s} R\cong R\otimes_{R^{w^{-1}sw}} R\otimes_R R_w, \quad R_w\otimes_R R_{w'}\cong R_{w'w},\]
for $w,w' \in C$ and $s\in \Pi_{[\lambda]}$. Combining with the description of Soergel bimodules in \cite[Theorem $24.12$]{EMTW}, any object of $\CB_{[\lambda]}$ (see Section $1.2.1$ for the definition)  is isomorphic to a direct summand of some object of the form:
\Eq{\label{singular Soergel bimod} R_w\otimes_{R^{J_1}} R^{I_2}\otimes_{R^{J_2}}\dots R^{I_{n-1}}\otimes_{R^{J_{n-1}}} R,}
where $[\emptyset \subset J_1\supset I_2\subset J_2\dots I_{n-1}\subset J_{n-1}\supset \emptyset]$ is a sequence of subsets of $\Pi_{[\lambda]}$.

 Let us recall the stabilizers $W_\mu, W_\lambda$ of $\mu, \lambda$ in $W$ under the dot action, respectively.  There are two subsets $I,J$ of $\Pi_{[\lambda]}$ such that $W_\mu=W_I$ and $W_\lambda=W_J$. Recall the decomposition $\widehat{W}_{[\lambda]}=C\ltimes W_{[\lambda]}$. For any $w\in C$, the weight $w\cdot \mu$ is dominant, and its stabilizer under the dot $W$-action is $wW_{\mu}w^{-1}$. Let $I_w=wIw^{-1}$. The action map $w: R^{W_\mu}\rightarrow R^{W_{w\cdot \mu}}$ equips $R^{W_{w\cdot\mu}}$ with  an $(R^{W_\mu}, R^{W_{w\cdot \mu}})$-bimodule structure as follows: $f\cdot g=w(f)g$ and $g\cdot h=hg$ for any $f\in R^{W_{\mu}}$ and $g,h\in R^{W_{w\cdot \mu}}$. We denote this bimodule by $R^{I_w}$.

Let us consider the following $(R^{W_{\mu}}, R^{W_{\lambda}})$-bimodule:
\Eq{\label{sing Soergel bimod 2} R^{I_w}\otimes_{R^{J_1}}\otimes R^{I_{2}}\otimes_{R^{J_2}}\dots R^{I_{n-1}} \otimes_{R^{J_{n-1}}}R^{W_\lambda},}
here $I_w\subset J_1\supset I_2\subset \dots I_{n-1}\subset J_{n-1}\supset J$. We equip the bimodule \eqref{sing Soergel bimod 2} with the $C(G)$-action corresponding to the lattice $w\cdot \mu-\lambda+\Lambda_r$. This makes bimodule \eqref{sing Soergel bimod 2} into an object in $\CB_{\mu, \lambda}$.

\begin{Prop}\label{lem:2rel Soergel cat}a) The category $\CB_{\mu, \lambda}$ is the smallest full additive subcategory of  the category $(R^{W_\mu},R^{W_\lambda})\bim^{C(G)}$ consisting of all objects isomorphic to direct summands of objects of the form \eqref{sing Soergel bimod 2} for some $w\in C$.

b) The number of indecomposable objects in $\CB_{\mu, \lambda}$ is equal to the number of double cosets $W_\mu\backslash \widehat{W}_{[\lambda]}/W_\lambda$.

\end{Prop}
% \begin{proof}This follows from the following observation: since $R_w\cong (R^{I_w})^{\oplus ?}$ as $(R^{W_\mu}, R^{W_{w\cdot \mu}})$-bimodules, the  object in \eqref{singular Soergel bimod} viewed as a $(R^{W_\mu}, R^{W_\lambda})$-bimodule is isomorphic to the following bimodule
% \[ (R^{I_w}\otimes_{R^{J_1}} R^{I_2}\otimes_{R^{J_2}}\dots R^{I_{n-1}}\otimes_{R^{J_{n-1}}} R\otimes_R R^{W_{\lambda}})^{\oplus ?}.\]
% \end{proof}
We will prove this lemma in Corollary \ref{Cor: 2rel of Soergel cat}. Let us now prove a weaker statement of Lemma \ref{lem:2rel Soergel cat}.b).

\begin{Lem}\label{cor number of singular Sorgel bimod} The number of indecomposable objects in $\CB_{\mu, \lambda}$ is greater than or equal to the number of double cosets $W_\mu \backslash \widehat{W}_{[\lambda]}/W_\lambda$.
\end{Lem}

\begin{proof}
    For any $ w\in C$, let us call indecomposable objects of $\CB_{\mu, \lambda}$ on which $C(G)$ acts via the character $w\cdot \mu -\lambda +\Lambda_r$  \emph{indecomposable objects of type $w$}. We see that the direct summands of some object in \eqref{sing Soergel bimod 2} are indecomposable objects of type $w$ .
    
    Under taking direct sums and direct summands, the objects in \eqref{sing Soergel bimod 2} viewed as $(R^{W_w\cdot\mu}, R^{W_\lambda})$-bimodules generate the category of singular Soergel bimodules $^{I_w}\CB^J$ attached to the Coxeter system $(W_{[\lambda]}, \Pi_{[\lambda]})$ and the pair $(I^w, J)$ of subsets in $\Pi_{[\lambda]}$; see \cite[]{GW} for details about the category of singular Soergel bimodules. By Theorem $1$ in \cite{GW}, the number of indecomposable objects in $^{I_w}\CB^J$ is equal to the number of double cosets $W_{w\cdot \mu}\backslash W_{[\lambda]}/W_\lambda$. On the other hand, the direct summands of the bimodules \eqref{sing Soergel bimod 2} viewed as   $(R^{W_\mu}, R^{W_\lambda})$-bimodules are the same as the direct summands of the bimodules \eqref{sing Soergel bimod 2} viewed as $(R^{W_{w\cdot \mu}}, R^{W_\lambda})$-bimodules. Thus, the number of indecomposable objects of type $w$ which are direct summands of objects in \eqref{sing Soergel bimod 2} is equal to the number of double cosets $W_{w\cdot \mu}\backslash W_{[\lambda]}/W_\lambda$. Therefore, the number of indecomposable objects in $\CB_{\mu, \lambda}$ is at least
    \[ \sum_{w\in C} \#W_{w\cdot \mu}\backslash W_{[\lambda]} /W_\lambda =\#W_\mu\backslash \widehat{W}_{[\lambda]}/W_\lambda.\]
    
\end{proof}

\subsection{Proof of Theorem $\ref{Thm4}$}\

Under the identifications $Z^{\wedge_\chi}\overset{\epsilon_\mu}{\cong} R^{W_\mu}$ and $Z^{\wedge_{\chi'}}\overset{\epsilon_\lambda}{\cong} R^{W_{\lambda}}$, the functor $\bullet_\dag$ in \eqref{complete restriction} becomes:
\[\bullet_\dag: \HC(U^{\chi, \chi'})\rightarrow (R^{W_{\mu}}, R^{W_{\lambda}})\bim^{C(G)}.\]

We now introduce a full Karoubian additive subcategory $\CP$ of the category $\HC(U^{\chi, \chi'})$. Recall the decomposition $\widehat{W}_{[\lambda]}=C\ltimes W_{[\lambda]}$ and the set of simple reflections $\Pi_{[\lambda]}=\{s_{\a^\lambda_1}, \dots, s_{\a^\lambda_k}\}$ in $W_{[\lambda]}$. For any $w\in C$, the weight $w \cdot \mu$ is dominant. By Lemma \ref{subgeneric weights}, there are dominant weights $\nu, \nu_i, 1\leq i \leq k$ in $\lambda +\Lambda$ such that $\nu$ is regular and $\text{Stab}_W(\nu_i)=\{1, s_{\a^\lambda_i}\}$. By abuse of notation, we will denote $s_{\a^\lambda_{i}}$ by $s_i$.

Let $I=\{i_1, \dots, i_l\}$ be a sequence of integers with $1\leq i_j \leq k$ for all $i_j$. To a pair $(w,I)$ with $w\in C$ and a sequence of integers $I$ as above, we define the following object  in $\HC(U^{\chi, \chi'})$:
\Eq{\label{P objects}P^{w, I}:=P^{w\cdot \mu, \nu} \otimes_{U^{\chi_\nu}}(P^{\nu, \nu_{i_1}}\otimes_{U^{\chi_{\nu_{i_1}}}}P^{\nu_{i_1}, \nu})\otimes_{U^{\chi_\nu}}\dots \otimes_{U^{\chi_{\nu}}}(P^{\nu, \nu_{i_l}}\otimes_{U^{\chi_{\nu_{i_l}}}} P^{\nu_{i_l}, \nu})\otimes_{U^{\chi_\nu}} P^{\nu, \lambda}.}

\begin{Def}
    Let $\CP$ be the full Karoubian additive subcategory of the category $\HC(U^{\chi, \chi'})$ consisting of all  direct sums of objects which are isomorphic to direct summands of some object of the form \eqref{P objects} for some $w\in C$.
\end{Def}
It is obvious that any object in $\CP$ is  projective  in $\HC(U^{\chi, \chi'})$. The next lemma concerns the image of $\CP$ under the functor $\bullet_\dag$.

\begin{Lem} Under the functor $\bullet_\dag$, the image $P^{w,I}_\dag $ is isomorphic to the bimodule:
\Eq{\label{image of P object} R_w\otimes_R (R\otimes_{R^{s_{i_1}}}R) \otimes_R \dots \otimes_R (R\otimes_{R^{s_{i_l}}} R)\otimes_R R}
in the category $\CB_{\mu, \lambda}$; the last component $R$ is viewed as a $(R,R^{W_\lambda})$-bimodule in the obvious way.
\end{Lem}
\begin{proof}
Using Proposition \ref{prop image of complete tranbimod}, if we use the identifications $Z^{\wedge_\chi} \overset{\epsilon_{w\cdot \mu}}{\cong} R^{W_{w\cdot \mu}}$ and $Z^{\wedge_{\chi'}}\overset{\epsilon_\lambda}{\cong} R^{W_\lambda}$, we can see that $P^{w,I}_\dag$ is isomorphic to :
\[ R\otimes_R (R\otimes_{R^{s_{i_1}}} R)\otimes_R \dots \otimes_R (R\otimes_{R^{s_{i_l}}} R)\otimes_R R\]
as $(R^{W_{w\cdot \mu}}, R^{W_\lambda})$-bimodules. So if we replace the identification $Z^{\wedge_\chi}\overset{\epsilon_{w\cdot \mu}}{\cong} R^{W_{w\cdot \mu}}$ by the identification $Z^{\wedge_\chi}\overset{\epsilon_\mu}{\cong}R^{W_\mu}$, we see that $P^{w, I}_\dag$ is isomorphic to \eqref{image of P object} as $(R^{W_\mu}, R^{W_\lambda})$-bimodules.

On the other hand, since $P^{w, I}$ is the direct summand of the bimodule:
\[(L(w\cdot \mu-\nu)\otimes L(\nu-\nu_{i_1})\otimes L(\nu_{i_1}-\nu)\otimes\dots \otimes L(\nu_{i_l}-\nu)\otimes L(\nu-\lambda)\otimes U(\g))\otimes_{\CZ}\CZ^{\chi, \chi'},\]
the group $C(G)$ acts on $P^{w,I}$ by the character corresponding to the coset $w\cdot \mu-\lambda +\Lambda_r$, hence $C(G)$ also acts on $P^{w,I}_\dag$ by that character.  
\end{proof}
The next two lemmas concern the homological properties of the category $\HC(U^{\chi, \chi'})$.
\begin{Lem}\label{cor20} Let $\mu,\lambda$ be two dominant weights such that $\mu\in \chi, \lambda \in\chi', \mu-\lambda\in \Lambda$ and $W_{\lambda}\subset W_\mu$. Then:
\[ P^{\mu, \lambda}\otimes_{U^{\chi'}} P^{\lambda, \mu} \cong (U^{\chi})^{\oplus|W_\mu/W_\lambda|}.\]
\end{Lem}

\begin{proof}We have $(P^{\mu, \lambda}\otimes_{U^{\chi'}}P^{\lambda, \mu})_\dag\cong (R^{W_\mu})^{\oplus |W_\mu/W_\lambda|}$ and $U^{\chi}_\dag \cong R^{W_{\mu}}$. On the other hand, the functor $\bullet_\dag: \HC(U^{\chi, \chi}) \rightarrow (R^{W_\mu}, R^{W_\mu})\bim^{C(G)}$ is fully faithful on projective objects. Hence, the lemma follows.
\end{proof}

\begin{Lem}\label{lem: fin-proj-resol} Any object in $\HC(U^{\chi, \chi'})$ has a finite resolution  by projective objects.
\end{Lem}
\begin{proof} First, let us assume $\chi'$ is regular. Then the lemma follows by Lemma \ref{prop finite homological O} and Proposition \ref{prop number of complete HCbim}. The proof for the case when $\chi$ is regular is the same by replacing the role of category $\hat{O}$ by the category $\hat{O}^r$, the analog  of $\hat{O}$ for right $U_R(\g)$-modules.

We now prove the lemma in general. Let $\nu$ be the regular dominant weight such that $\nu \in \lambda+  \Lambda$. By Corollary $\ref{cor20}$, we have:
\Eq{\label{eq32} P^{\mu, \nu}\otimes_{U^{\chi_\nu}}P^{\nu, \mu} \cong (U^{\chi_\nu})^{\oplus|W_{\mu}|}.}
Hence, for any $M\in \HC(U^{\chi, \chi'})$, we have:
\[ M^{\oplus |W_\mu|} \cong P^{\mu, \nu}\otimes_{U^{\chi_\nu}}P^{\nu, \mu}\otimes_{U^\chi} M.\]
Note that the  functor $P^{\mu, \nu}\otimes_{U^{\chi_\nu}}\bullet: \HC(U^{\chi_\nu, \chi'}) \rightarrow \HC(U^{\chi,\chi'})$ is exact and maps projective objects to projective objects. Therefore, $M^{\oplus|W_\mu|}$  has a finite homological dimension, hence, so does $M$.
\end{proof}
We will need the following categorical lemma in the proof of Theorem $\ref{Thm4}$.

\begin{Lem} \label{categorical lem}
    Let $\CA$ and $\mathcal{C}$ be two Karoubian additive categories. Let $F: \CA\rightarrow \mathcal{C}$ be an additive functor such that: 
    \begin{itemize}
        \item $F$ is fully faithful.
        \item For any $C\in \mathcal{C}$, there are $A\in \CA$ and $D\in \mathcal{C}$ such that $F(A)\cong C\oplus D$.
    \end{itemize}
    Then $F$ is an equivalence of Karoubian additive categories.
\end{Lem}
% \begin{proof} 
%     It is enough to show that $F$ is essentially surjective. Let $C$ be an object in $\mathcal{C}$ then there are $A \in \CA$ and $D\in \mathcal{C}$  such that $F(A)\cong C\oplus D$. Let $i_D: F(A)\rightarrow F(A)$ be the corresponding projection onto $D$.

%     Since $F$ is fully faithful, we have an isomorphism of abelian groups
%     \[ \End_\CA(A,A)\xrightarrow[]{\cong} \End_{\mathcal{C}}(F(A),F(A)).\]
%     Let $\tilde{i}_D: A\rightarrow A$ be the corresponding idempotent map of $i_D$ then $A\cong \text{ker}(\tilde{i}_D)\oplus \text{coker}(\tilde{i}_D)$. Therefore, $F(\text{ker}(\tilde{i}_D))$ is also the kernel of the map $i_D: F(A)\rightarrow F(A)$. This implies that $F(\text{ker}(\tilde{i}_D))\cong C$.
% \end{proof}

\begin{proof}[Proof of Theorem $\ref{Thm4}$] The images of objects in $\CP$ under the functor $\bullet_\dag$ are contained in the category $\CB_{\mu, \lambda}$. So by Proposition \ref{prop ff of complete proj}, we have a fully faithful functor between two Karoubian additive categories $\bullet_\dag: \CP\rightarrow \CB_{\mu, \lambda}$. It is easy to see  that the functor $\bullet_\dag: \CP\rightarrow \CB_{\mu, \lambda}$ satisfies the conditions in Lemma \ref{categorical lem}. Therefore, we have an equivalence of Karoubian additive categories:
\Eq{\label{eq21} \CP\cong \CB_{\mu, \lambda}.}

By Lemma \ref{cor number of singular Sorgel bimod},  the number of indecomposable objects in $\CP$ is greater than or equal to  the number of  double cosets $W_\mu\backslash \widehat{W}_{[\lambda]}/W_\lambda$. The latter number is equal to the number of indecomposable projective objects in $\HC(U^{\chi, \chi'})$ by Proposition \ref{prop number of complete HCbim}. Therefore, the number of indecomposable objects in $\CP$ is equal to the number of indecomposable projective objects in $\HC(U^{\chi, \chi'})$. This implies that $\CP$ must be the full subcategory  consisting of all projective objects in  $\HC(U^{\chi, \chi'})$. 

On the other hand, by Lemma \ref{lem: fin-proj-resol}, any object in $\HC(U^{\chi, \chi'})$ has a finite resolution by projective objects. Therefore, 
\[D^b(\HC(U^{\mu,\lambda}))\cong K^b(\CP) .\]
Combining with the equivalence (\ref{eq21}), we have:  
\[ D^b(\HC(U^{\chi, \chi'}))\cong K^b(\CB_{\mu, \lambda})=: \widehat{\CH}.\]
\end{proof}

\begin{Cor}\label{Cor: 2rel of Soergel cat} Proposition \ref{lem:2rel Soergel cat} holds.
\end{Cor}
\begin{proof} By $\CP \cong \CB_{\mu, \lambda}$, the number of indecomposable objects in $\CB_{\mu, \lambda}$ is equal to the number of double cosets $W_\mu \backslash \widehat{W}_{[\lambda]} /W_\lambda$, so part $b)$ holds. Therefore, we can produce all indecomposable objects in $\CB_{\mu, \lambda}$ by taking direct summands of objects of the form \eqref{sing Soergel bimod 2}, hence, part $a)$ holds.
\end{proof}

\printbibliography

@book {H08,
    AUTHOR = {Humphreys, J. E.},
     TITLE = {Representations of semisimple {L}ie algebras in the {BGG}
              category $O$},
    SERIES = {Graduate Studies in Mathematics},
    VOLUME = {94},
 PUBLISHER = {American Mathematical Society, Providence, RI},
      YEAR = {2008},
     PAGES = {xvi+289},
      ISBN = {978-0-8218-4678-0},
   MRCLASS = {17B10},
  MRNUMBER = {2428237},
MRREVIEWER = {Serge\ M.\ Skryabin},
       DOI = {10.1090/gsm/094},
       URL = {},
}

@article {K63,
    AUTHOR = {Kostant, B. },
     TITLE = {Lie group representations on polynomial rings},
   JOURNAL = {Amer. J. Math.},
  FJOURNAL = {American Journal of Mathematics},
    VOLUME = {85},
      YEAR = {1963},
     PAGES = {327--404},
      ISSN = {0002-9327,1080-6377},
   MRCLASS = {22.60 (20.80)},
  MRNUMBER = {158024},
MRREVIEWER = {J.\ L.\ Koszul},
       DOI = {10.2307/2373130},
       URL = {},
}

@article{IL15,
author = {Losev, I.},
title = {{Dimensions of irreducible modules over W-algebras and Goldie ranks}},
volume = {200},
journal = {Invent. math},
number = {1},
publisher = {Duke University Press},
pages = {849 -- 923},
year = {2015},
doi = {10.1007/s00222-014-0541-0},
URL = {}
}

@article {IL11,
    AUTHOR = {Losev, I. },
     TITLE = {Finite-dimensional representations of {$W$}-algebras},
   JOURNAL = {Duke Math. J.},
  FJOURNAL = {Duke Mathematical Journal},
    VOLUME = {159},
      YEAR = {2011},
    NUMBER = {1},
     PAGES = {99--143},
      ISSN = {0012-7094,1547-7398},
   MRCLASS = {17B35 (17B37)},
  MRNUMBER = {2817650},
MRREVIEWER = {Simon\ M.\ Goodwin},
       DOI = {10.1215/00127094-1384800},
       URL = {},
}

@article {So92,
    AUTHOR = {Soergel, W.},
     TITLE = {The combinatorics of {H}arish-{C}handra bimodules},
   JOURNAL = {J. Reine Angew. Math.},
  FJOURNAL = {Journal f\"ur die Reine und Angewandte Mathematik. [Crelle's
              Journal]},
    VOLUME = {429},
      YEAR = {1992},
     PAGES = {49--74},
      ISSN = {0075-4102,1435-5345},
   MRCLASS = {17B10 (22E47)},
  MRNUMBER = {1173115},
MRREVIEWER = {Ronald\ S.\ Irving},
       DOI = {10.1515/crll.1992.429.49},
       URL = {},
}

@article {BG80,
    AUTHOR = {Bernstein, J. N. and Gelfand, S. I.},
     TITLE = {Tensor products of finite- and infinite-dimensional
              representations of semisimple {L}ie algebras},
   JOURNAL = {Compositio Math.},
  FJOURNAL = {Compositio Mathematica},
    VOLUME = {41},
      YEAR = {1980},
    NUMBER = {2},
     PAGES = {245--285},
      ISSN = {0010-437X,1570-5846},
   MRCLASS = {17B10 (22E47)},
  MRNUMBER = {581584},
MRREVIEWER = {A.\ U.\ Klimyk},
       URL = {},
}

@article {STL,
    AUTHOR = {Stroppel, C.},
     TITLE = {A structure theorem for {H}arish-{C}handra bimodules via
              coinvariants and {G}olod rings},
   JOURNAL = {J. Algebra},
  FJOURNAL = {Journal of Algebra},
    VOLUME = {282},
      YEAR = {2004},
    NUMBER = {1},
     PAGES = {349--367},
      ISSN = {0021-8693,1090-266X},
   MRCLASS = {17B10 (17B20)},
  MRNUMBER = {2097586},
MRREVIEWER = {Gordon\ E.\ Brown},
       DOI = {10.1016/j.jalgebra.2004.07.037},
       URL = {https://doi.org/10.1016/j.jalgebra.2004.07.037},
}

@article {GW,
    AUTHOR = {Williamson, G. },
     TITLE = {Singular {S}oergel bimodules},
   JOURNAL = {Int. Math. Res. Not. IMRN},
  FJOURNAL = {International Mathematics Research Notices. IMRN},
      YEAR = {2011},
    NUMBER = {20},
     PAGES = {4555--4632},
      ISSN = {1073-7928,1687-0247},
   MRCLASS = {20F55 (18D05)},
  MRNUMBER = {2844932},
MRREVIEWER = {G\"otz\ Pfeiffer},
       DOI = {10.1093/imrn/rnq263},
       URL = {},
}

@article {R07,
    AUTHOR = {Rouquier, R. },
     TITLE = {{$q$}-{S}chur algebras and complex reflection groups},
   JOURNAL = {Mosc. Math. J.},
  FJOURNAL = {Moscow Mathematical Journal},
    VOLUME = {8},
      YEAR = {2008},
    NUMBER = {1},
     PAGES = {119--158, 184},
      ISSN = {1609-3321,1609-4514},
   MRCLASS = {20G43 (16G30 20C08 20F55)},
  MRNUMBER = {2422270},
MRREVIEWER = {Andrew\ Mathas},
       DOI = {10.17323/1609-4514-2008-8-1-119-158},
       URL = {},
}

@article {LY20,
    AUTHOR = {Lusztig, G. and Yun, Z.},
     TITLE = {Endoscopy for {H}ecke categories, character sheaves and
              representations},
   JOURNAL = {Forum Math. Pi},
  FJOURNAL = {Forum of Mathematics. Pi},
    VOLUME = {8},
      YEAR = {2020},
     PAGES = {e12, 93},
      ISSN = {2050-5086},
   MRCLASS = {20G40 (14F08 20C08 20C33)},
  MRNUMBER = {4108915},
MRREVIEWER = {Karin\ Erdmann},
       DOI = {10.1017/fmp.2020.9},
       URL = {},
}

@book {EMTW,
    AUTHOR = {Elias, Ben and Makisumi, Shotaro and Thiel, Ulrich and
              Williamson, Geordie},
     TITLE = {Introduction to {S}oergel bimodules},
    SERIES = {RSME Springer Series},
    VOLUME = {5},
 PUBLISHER = {Springer, Cham},
      YEAR = {2020},
     PAGES = {xxv+588},
      ISBN = {978-3-030-48825-3},
   MRCLASS = {20C08 (17B10 18M30 20F55)},
  MRNUMBER = {4220642},
MRREVIEWER = {G\"otz\ Pfeiffer},
       DOI = {10.1007/978-3-030-48826-0},
       URL = {},
}

\end{document}